\documentclass[reqno,11pt,a4paper]{amsart}
\usepackage{amsmath,amssymb,amsthm,graphicx,mathrsfs,url}
\usepackage[usenames,dvipsnames]{xcolor}
\usepackage[colorlinks=true,linkcolor=Red,citecolor=Green]{hyperref}
\usepackage{tikz}
\usetikzlibrary{decorations.pathreplacing}
\usetikzlibrary{arrows}
\usetikzlibrary{shapes.misc}
\usetikzlibrary{shapes.symbols}
\usetikzlibrary{patterns}

\def\?[#1]{\textbf{[#1]}\marginpar{\Large{\textbf{??}}}}

\calclayout

\newtheorem{theo}{Theorem}
\newtheorem{prop}{Proposition}[section]

\newtheorem{lemm}[prop]{Lemma}

\numberwithin{equation}{section}

\newcommand{\mc}{\mathcal}
\newcommand{\rr}{\mathbb{R}}
\newcommand{\nn}{\mathbb{N}}
\newcommand{\cc}{\mathbb{C}}

\newcommand{\la}{\lambda}
\newcommand{\eps}{\varepsilon}

\newcommand{\pl}{\partial}

\newcommand{\til}{\widetilde}
\newcommand{\bbar}{\overline}

\newcommand{\cjd}{\rangle}
\newcommand{\cjg}{\langle}

\DeclareMathOperator{\comp}{comp}

\DeclareMathOperator{\Ell}{ell}

\let\Im=\Imag

\DeclareMathOperator{\Op}{Op}

\let\Re=\Real

\DeclareMathOperator{\supp}{supp}

\DeclareMathOperator{\divv}{div}

\DeclareMathOperator{\M}{\mathcal{M}}

\title[First band of Ruelle resonances in dimension $3$]{First band of Ruelle resonances for contact Anosov flows in dimension $3$}
\author[M. Ceki\'c]{Mihajlo Ceki\'c}
\email{mihajlo.cekic@universite-paris-saclay.fr}
\address{Laboratoire de Math\'ematiques d'Orsay, Universit\'e Paris-Saclay, CNRS, 91405 Orsay, France}
\author[C. Guillarmou]{Colin Guillarmou}
\email{colin.guillarmou@math.u-psud.fr}
\address{Laboratoire de Math\'ematiques d'Orsay, Univ. Paris-Sud, CNRS, Universit\'e Paris-Saclay, 91405 Orsay, France}

\begin{document}
\maketitle

\begin{abstract}
We show, using semiclassical measures and unstable derivatives, that a smooth vector field  
$X$ generating a contact Anosov flow on a $3$-dimensional manifold $\mc{M}$ has only finitely many Ruelle resonances in the vertical strips $\{ s\in \cc\ |\ {\rm Re}(s)\in [-\nu_{\min}+\eps,-\frac{1}{2}\nu_{\max}-\eps]\cup [-\frac{1}{2}\nu_{\min}+\eps,0]\}$ for all $\eps>0$, where $0<\nu_{\min}\leq \nu_{\max}$ are the minimal and maximal expansion rates of the flow
(the first strip only makes sense if $\nu_{\min}>\nu_{\max}/2$). 
We also show polynomial bounds in $s$ for the resolvent $(-X-s)^{-1}$ as $|{\rm Im}(s)|\to \infty$ in Sobolev spaces, and obtain similar results for cases with a potential. This is a short proof of a particular case of the results by Faure-Tsujii in \cite{FaTs1,FaTs2,FaTs3}, using that $\dim E_u=\dim E_s=1$. 
\end{abstract}

\section{Introduction}

In this note, we study the localization of Ruelle resonances for contact Anosov flows in dimension $3$ using semiclassical measures, and we show  the existence of a first band of resonances under a pinching condition on the maximal and minimal expansion rates of the flow. The fact that a vector field $X$ generating a contact Anosov flow has a band structure (see Figure \ref{fig:bands}) for its Ruelle resonance spectrum is proved in a series of seminal papers by Faure-Tsujii \cite{FaTs1,FaTs2,FaTs3} using FBI transform techniques and normal forms. In the $3$-dimensional case, we give a rather short proof of the existence of the first band using unstable derivatives and semiclassical measures in the spirit of Dyatlov's proof \cite{Dy} for operators with $r$-normally hyperbolic trapped set. Thus the aim of this work is to use the $3$-dimensional particularity and semiclassical measures to present the mechanism behind this structure.  

Let $\mc{M}$ be a $3$-dimensional closed manifold, equipped with a Riemannian metric $g$ and let $X$ be a smooth vector field on $\mc{M}$ such that its flow $\varphi_t$ is Anosov (see Section \ref{sec:dynback} for a definition). Denote by $E_u\subset T\mc{M}$ and $E_s\subset T\mc{M}$ the unstable and stable bundles. We shall assume that $\varphi_t$ is a contact Anosov flow, which means that there is a smooth $1$-form $\alpha$ on $\mc{M}$ such that $\ker \alpha=E_u\oplus E_s$, $\alpha(X)=1$ and $\alpha\wedge d\alpha$ is a volume form, and that $E_u,E_s$ are orientable\footnote{For instance, this is satisfied for Anosov geodesic flows on orientable Riemannian surfaces.} (and hence trivializable) bundles.
Define the minimal and maximal expansion rate of $\varphi_t$ to be\footnote{The limit exists by Fekete's lemma, and the same holds as well for $V_{\max},V_{\min}$ defined later.} 
\[\nu_{\min} := \lim_{t \to \infty} \inf_{x \in \mc{M}} \frac{1}{t}
\log\|d\varphi_{t}(x)|_{E_u}\|_g,\quad 
	\nu_{\max} := \lim_{t \to \infty} \sup_{x \in \mc{M}} \frac{1}{t}\log\|d\varphi_{t}(x)|_{E_u}\|_g.\]
Now, we come to the notion of Ruelle resonances for $X$. It is proved by Butterley-Liverani \cite{BuLi} and Faure-Sj\"ostrand \cite{FaSj} that for each $N>0$, there is a Hilbert space $\mc{H}^{N}$, called anisotropic Sobolev space of order $N$, satisfying\footnote{$H^N(\mc{M})=(1+\Delta_g)^{-N/2}L^2(\mc{M})$ denotes the usual Sobolev space of order $N$ on $\mc{M}$.} $H^N(\mc{M})\subset \mc{H}^{N}\subset H^{-N}(\mc{M})$ so that for ${\rm Re}(s)>-\nu_{\min}N$
\[ -X-s: \{ u\in \mc{H}^N\, |\, -Xu\in \mc{H}^N\}\to \mc{H}^N\]
is a Fredholm operator of index $0$, implying that $-X$ has discrete spectrum in the half-plane 
$\{s \in \cc\,|\, {\rm Re}(s)>-\nu_{\min}N\}$. Moreover, there is no spectrum in ${\rm Re}(s)>0$,
the spectrum in ${\rm Re}(s)>-\nu_{\min}N_0$ on $\mc{H}^N$ does not depend on $N\geq N_0$ or the choice of the space $\mc{H}^N$ and the resolvent 
\[ R(s):=(-X-s)^{-1}: C^\infty(\mc{M})\to C^{-\infty}(\mc{M})\]  
is meromorphic in $\cc$. An $s \in \mathbb{C}$ is a pole of the resolvent if and only if it is an eigenvalue of $-X$ in $\mc{H}^N$ for some $N>-{\rm Re}(s)/\nu_{\min}$, and such $s$ are called \emph{Ruelle resonances}, and $u \in \mc{H}^N$ with $(X + s)u = 0$ are called \emph{resonant states}.

Since we want to prove absence of Ruelle resonances and bounds on the resolvent $R(s)$ for large ${\rm Im}(s)$, it is convenient to use a semiclassical rescaling, that is to take a small 
parameter $h>0$ so that $h{\rm Im}(s)$ is uniformly bounded. 
For such small parameter $h>0$, we denote by $H_h^N$ the $H^N(\mc{M})$ Sobolev space equipped with the norm $\|u\|_{H_h^N} := \|(1+h^2\Delta_g)^{N/2}u\|_{L^2}$. 

We prove the following result:
\begin{theo}\label{th:intro}
Assume that $X$ is a smooth contact Anosov flow on a closed $3$-manifold $\mc{M}$ with $E_u,E_s$ orientable, and let $\nu_{\min},\nu_{\max}$ be the minimal and maximal expansion rates of the flow. 
For any $\varepsilon > 0$, there are only finitely many Ruelle resonances in the regions (see Figure \ref{fig:bands})
\[\mc{S}_0(\eps):=\Big\{s\in \cc \, | \, \Re s >-\frac{\nu_{\min}}{2} + \eps\Big\} \quad and \quad \mc{S}_1(\eps) := \Big\{s\in \cc \, | \, \Re s \in (-\nu_{\min}+\varepsilon, -\frac{\nu_{\max}}{2} - \eps)\Big\},
\]
and the following bounds hold, as $|s| \to \infty$
\begin{align*}
	\|(-X - s)^{-1}\|_{H^{\frac{1}{2}}(\mc{M}) \to H^{-\frac{1}{2}}(\mc{M})} &= o(|s|), \quad \,\,\,\,s\in \mc{S}_0(\varepsilon),\\
	\|(-X - s)^{-1}\|_{{H^{1}}(\mc{M}) \to H^{-1}(\mc{M})} &= \mc{O}(|s|), \quad \,\,s\in \mc{S}_1(\varepsilon).
\end{align*}
More precisely, if $h\in (0,h_0)$ is a small parameter, for any $N > 0$ there exist Hilbert spaces $\mc{H}_h^{N}$ satisfying  uniform bounds 
\[ \exists C>0, \forall u\in C^\infty(\mc{M}), \forall h\in (0,h_0),\quad  C^{-1}\|u\|_{H_h^{-N}}\leq \|u\|_{\mc{H}_h^{N}}\leq C\|u\|_{H_h^N},\]  
such that the following bounds hold: for all $\eps>0$ and $A>1$, and $\lambda\in \cc$ satisfying $|{\rm Im}(\lambda)|\in [A^{-1},A]$, as $h \to 0$
\begin{align*}
	\|(-hX - \lambda)^{-1}\|_{\mc{H}_h^{N} \to \mc{H}_h^{N}} &= o(h^{-2}),\quad \,\,\,\,h^{-1}\lambda \in \mc{S}_0(\varepsilon), \quad N \geq \frac{1}{2},\\
	\|(-hX - \lambda)^{-1}\|_{\mc{H}_h^{N} \to \mc{H}_h^{N}} &= \mc{O}(h^{-2}), \quad \,\,h^{-1}\lambda \in \mc{S}_1(\varepsilon), \quad N \geq 1.
\end{align*}
\end{theo}

		\begin{figure}
             \centering
\begin{tikzpicture}[scale = 0.8, everynode/.style={scale=0.5}]
\tikzset{cross/.style={cross out, draw=black, minimum size=2*(#1-\pgflinewidth), inner sep=0pt, outer sep=0pt},
cross/.default={1pt}}

      	\fill[gray!30] (7.8, -3) rectangle (8.5, 3); 
      	\fill[gray!30] (5.5, -3) rectangle (7, 3); 

      	\draw[thick, ->] (1,-3) -- (1,3) node[above] {\small $\Im s$};
		\draw[thick, ->] (-3.5,0) -- (3.5,0) node[below] {\small $\Re s$};
      	
      	\draw[thick] (-2, -3) -- (-2, 3); 
      	\draw[thick] (-3.5, -3) node[below] {\tiny $\Re(s) =$};
      	\draw[thick] (-2, -3) node[below] {\tiny $-\frac{3}{2}$};
		\draw[thick] (-2, 0) node[cross=3pt,rotate=0,red]{}; 
		\draw[thick] (-2.3, 0) node[cross=3pt,rotate=0,blue]{};
		\draw[thick] (-1.7, 0) node[cross=3pt,rotate=0,blue]{};
		\draw[thick] (-2, 0) node[cross=3pt,rotate=0,red]{}; 
		
		\draw[thick] (-2, 1) node[cross=3pt,rotate=0,green]{}; 
    	\draw[thick] (-2, 1.5) node[cross=3pt,rotate=0,green]{}; 
    	\draw[thick] (-2, 2.3) node[cross=3pt,rotate=0,green]{}; 
    	\draw[thick] (-2, -1) node[cross=3pt,rotate=0,green]{}; 
    	\draw[thick] (-2, -1.5) node[cross=3pt,rotate=0,green]{}; 
    	\draw[thick] (-2, -2.3) node[cross=3pt,rotate=0,green]{}; 
    	
      	\draw[thick] (0, -3) node[below] {\tiny $-\frac{1}{2}$} -- (0, 3);
		\draw[thick] (1, 0) node[cross=3pt,rotate=0,blue]{} node[above right]{\small $0$};
		\draw[thick] (0, 0) node[cross=3pt,rotate=0,red]{};
		\draw[thick] (0.3, 0) node[cross=3pt,rotate=0,blue]{};
		\draw[thick] (-0.3, 0) node[cross=3pt,rotate=0,blue]{};
		
		\draw[thick] (0, 1) node[cross=3pt,rotate=0,green]{}; 
    	\draw[thick] (0, 1.5) node[cross=3pt,rotate=0,green]{}; 
    	\draw[thick] (0, 2.3) node[cross=3pt,rotate=0,green]{}; 
    	\draw[thick] (0, -1) node[cross=3pt,rotate=0,green]{}; 
    	\draw[thick] (0, -1.5) node[cross=3pt,rotate=0,green]{}; 
    	\draw[thick] (0, -2.3) node[cross=3pt,rotate=0,green]{}; 

		\draw[thick] (-1, 0) node[cross=3pt,rotate=0,red]{} node[below]{\small $-1$};


		\draw[thick] (1.8, 2) node[cross=3pt,rotate=0,green]{} node[right]{\tiny $= \mathrm{large\, eigenvalues}$};
		\draw[thick] (1.8, 1.7) node[cross=3pt,rotate=0,blue]{} node[right]{\tiny $= \mathrm{small\, eigenvalues}$};
		\draw[thick] (1.8, 1.4) node[cross=3pt,rotate=0,red]{} node[right]{\tiny $= \mathrm{special\, points}$};

		
		\draw[thick, ->] (5.5,0) -- (12.5,0) node[below] {\small $\Re s$};
      	\draw[thick, ->] (10,-3) -- (10,3) node[above] {\small $\Im s$};
      	
      	
      	\draw[dashed] (8.5, -3) -- (8.5, 3); 
      	\draw[dashed] (7.8, -3) -- (7.8, 3);
      	\draw[dashed] (7, -3) -- (7, 3); 
      	\draw (9, -3) node[below] {\tiny $-\frac{\nu_{\min}}{2}$};
      	\draw (7.7, -3) node[below] {\tiny $-\frac{\nu_{\max}}{2}$};
      	\draw (6.4,-3) node[below] {\tiny $-\nu_{\min}$};
      	\draw (5, -3) node[below] {\tiny $\Re (s) = $};
      	\fill[pattern=north east lines, pattern color=blue] (8.5, -3) rectangle (10, 3); 
      	\fill[pattern=north east lines, pattern color=blue] (7, -3) rectangle (7.8, 3); 
      	\fill[pattern=north east lines, pattern color=blue] (7, -3) rectangle (7.8, 3); 

    \node[font=\boldmath\fontsize{2.5mm}{3.5mm}\selectfont] at (7.4, 1.5) {$\mathcal{O}(h^{-2})$};
    \node[font=\boldmath\fontsize{4mm}{5mm}\selectfont] at (9.25, 1.5) {$o(h^{-2})$};

\end{tikzpicture}
             \caption{\small Left: resonance spectrum of a closed hyperbolic surface $\Sigma$ \cite{DFG}; colouring describes the nature of the resonance. Right: band structure of a contact Anosov flow on a $3$-manifold given by Theorem \ref{th:intro}. In dashed blue regions we have only finitely many resonances, with semiclassical resolvent bounds written using $\mathcal{O}$ and $o$ notation; gray regions are not covered by the theorem.}
             \label{fig:bands}
\end{figure}
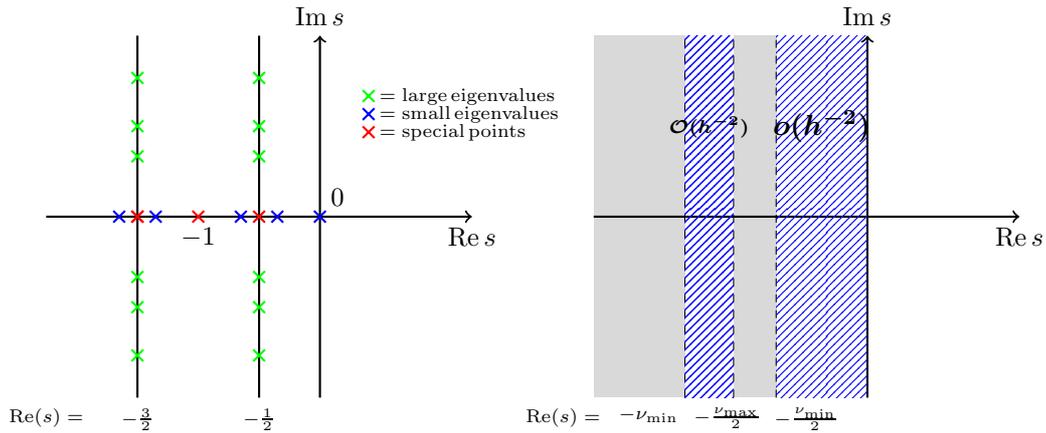

As mentioned above, this result is contained in the set of results of Faure-Tsujii \cite{FaTs1,FaTs2,FaTs3}. In fact, the second strip we obtain with finitely many resonances is not optimal and should be ${\rm Re}(s)\in [-\frac{3\nu_{\min}}{2}+\eps,-\frac{\nu_{\max}}{2}-\eps]$ for all $\eps>0$ small. A band structure for Ruelle resonances was first obtained by Faure-Tsujii in \cite{FaTs0,FaTs2, FaTs3} for prequantum maps and Anosov flows using FBI transform methods and normal forms, including the case of a geometric potential (multiples of the unstable Jacobian). In the case of constant curvature or locally symmetric spaces, the bands are actually lines \cite{FaTs1,DFG,GHW,KuWe} (see Figure \ref{fig:bands}). In a related context, for operators with $r$-normally hyperbolic trapped set, Dyatlov \cite{Dy0,Dy} also proved some band structure for resonances.

Our proof is adapted to the case of a $1$-dimensional bundles $E_u,E_s$, and should be quite accessible for a reader with a semiclassical background (pseudo-differential calculus).
The aim of this note is to exhibit the mechanism behind the band structure in this simpler $3$-dimensional case. 
The existence of the first region $\{s\in \cc\ |\ {\rm Re}(s)>-\nu_{\min}/2+\eps\}$ with only finitely many Ruelle resonances for contact Anosov flows was first shown by Tsujii \cite{Ts1,Ts2} using FBI transform, and then by Nonnenmacher-Zworski \cite{NoZw} using normal hyperbolicity of the trapped set; both results also hold in higher dimension where 
$\nu_{\min}$ is replaced by $\lim_{t \to \infty} \inf_{x \in \mc{M}} \frac{1}{t}
\log \det(d\varphi_{t}(x)|_{E_u})$. Here, the method of proof is more similar to Dyatlov's approach \cite{Dy0,Dy} for $r$-normally hyperbolic trapped sets, and uses semiclassical measures and propagation estimates. It was also used in the setting of frame flows in constant curvature \cite{GK}. We also mention that the existence of a (non-explicit) small strip without Ruelle resonances in the contact Anosov setting was first proved by Dolgopyat \cite{Do} and Liverani \cite{Li}, and that there are two recent breakthroughs for $3$-dimensional Anosov flows: first for volume preserving Anosov flows by Tsujii \cite{Ts3}, and finally Tsujii-Zhang \cite{TsZh} recently proved the existence of a strip with no Ruelle resonances for all topologically mixing Anosov flows in dimension $3$.

The main idea of our proof is based on an observation made in a paper of the second author and Faure \cite{FaGu}, namely that the Ruelle resonant states $u\in \mc{H}^1$ with resonance in ${\rm Re}(s)>-\nu_{\min}$ satisfy $U_-u=0$ where $U_-$ is a vector field tangent to the unstable foliation (with regularity $C^{2-\eps}(\mc{M}; T\M)$). The outline of our proof is as follows. First, the existence of a sequence of Ruelle resonances $s_n$ (or of quasimodes) satisfying ${\rm Re}(s_n)\to -\gamma$ for some $-\gamma>-\nu_{\min}+\eps$ and ${\rm Im}(s_n)\to \infty$ implies the existence of a non-zero semiclassical measure $\mu$ on $T^*\mc{M}$, which, in microlocal terms, is the weak limit of a sequence of Ruelle resonant states (or quasimodes) $u_n\in \mc{H}^1$. By microlocal ellipticity it is supported in $\{(x,\xi)\in T^*\mc{M}\,|\, \xi(X)=-1\}$, 
and using the propagation of singularities estimate, one can see that for all $t\in \rr$
\begin{equation}\label{intro:invmuPhi} 
(\Phi_t)_*\mu=e^{-2\gamma t}\mu,
\end{equation}
where $\Phi_t(x,\xi)=(\varphi_t(x),(d\varphi_t(x)^{-1})^T\xi)$ is the symplectic lift of the flow $\varphi_t$ on $T^*\mc{M}$.
Next, using radial point propagation estimates related to the hyperbolicity of $\Phi_t$, we can see that 
\[{\rm supp}(\mu) \subset \Gamma_+:= E_u^*\oplus E_0^*,\] 
where $E_0^*:=\rr \alpha$ and $E_u^*$ is the annihilator of $\rr X \oplus E_u\subset T\mc{M}$. The last step is to use a version of the horocyclic invariance of the resonant states $u_n$ proved in \cite{FaGu}, that is $U_-u_n=0$ (or small in $h$). This invariance of $u_n$ by $U_-$ can be understood as an extra invariance of the measure $\mu$, which implies (by propagation estimates again) some regularity of $\mu$ at the trapped set $E_0^*$ of $\Phi_t$. This is where the contact assumption is important, it appears as a transversality of the symplectic lift of the horocyclic flow (i.e. the flow of $U_-$) with the trapped set $E_0^*=\rr\alpha$. Namely one shows that for all $\delta>0$ small and $\mc{U}_\delta$ a $\delta$-neighborhood of $E_0^*\cap \{\xi(X)=-1\}$ in $T^*\mc{M}$ 
\[ \mu(\mc{U}_\delta)\in (\delta/C,C\delta)\]
for some uniform $C>0$. Combining with \eqref{intro:invmuPhi} and using the hyperbolicity of $\Phi_t$:
\[\forall t\geq 0, \quad (\mc{U}_{\delta e^{-(\nu_{\max}+\eps)t}}\cap \Gamma_+) \subset  \Phi_{-t}(\mc{U}_\delta \cap \Gamma_+) \subset (\mc{U}_{\delta e^{-(\nu_{\min}-\eps)t}}
\cap \Gamma_+),\]
which implies by an elementary argument that $-\gamma$ must be in $[-(\nu_{\max} + \eps)/2,-(\nu_{\min} - \eps)/2]$. 
One of the main issues in this argument is that $U_-$ is only a $C^{2-\eps}(\M; T\M)$ vector field for all $\eps>0$, which complicates the use of microlocal methods. To circumvent the problem, we regularize $U_-$ at scale $h^\rho=|{\rm Im}(s_n)|^{-\rho}$ for some $0 < \rho < 1$, to make it a smooth, $h$-dependent, vector field $U_-^h$ in some exotic class of differential operators. The resonant states (or quasimodes) are not killed anymore by $U_-^h$ but they are good enough quasimodes to apply the reasoning above. We have to show, in particular, that the propagation estimates still hold in this exotic class (see Appendix \ref{app:exoticcalc}).\\  

In fact, since it does not involve more difficulties, we prove the more general result involving a smooth real potential $V$. We introduce the quantities
\begin{equation}\label{eq:V_maxV_min}
	V_{\max}:= \lim_{t \to \infty} \sup_{x \in \mc{M}} \frac{1}{t}\int_0^t V(\varphi_{s}(x))ds, \quad V_{\min}:=\lim_{t \to \infty} \inf_{x \in \mc{M}} \frac{1}{t}\int_0^t V(\varphi_{s}(x))ds.
\end{equation}
\begin{theo}\label{th:intro2}
Let $\mc{M}$ and $X$ satisfy the same assumptions as in Theorem \ref{th:intro},
and let $V\in C^\infty(\mc{M})$ be a real valued potential. For any $\varepsilon > 0$, there are only finitely many Ruelle resonances of $-X + V$ in the regions 
\[\begin{gathered}
\mc{S}_0(\eps):=\Big\{s\in \cc \, | \, \Re s >-\frac{\nu_{\min}}{2}+V_{\max} + \eps\Big\} \,\,\,\,\textrm{ and }\\
\mc{S}_1(\varepsilon) := \Big\{s\in \cc \, | \, \Re s \in (-\nu_{\min}+V_{\max}+\eps, -\frac{\nu_{\max}}{2}+V_{\min}-\eps)\Big\},
\end{gathered}\]
and the following bounds hold, as $|s| \to \infty$
\begin{align}\label{eq:classicalbounds}
\begin{split}
	\|(-X + V - s)^{-1}\|_{H^{\frac{1}{2}}(\mc{M}) \to H^{-\frac{1}{2}}(\mc{M})} &= o(|s|),\quad \,\,\,\,s\in \mc{S}_0(\varepsilon),\\
	\|(-X + V - s)^{-1}\|_{{H^{1}}(\mc{M}) \to H^{-1}(\mc{M})} &= \mc{O}(|s|), \quad \,\,s\in \mc{S}_1(\varepsilon).
\end{split}
\end{align}
More precisely, with the notation of Theorem \ref{th:intro}, then for all $\eps>0$ and $A>1$, and for all $\lambda \in \cc$ satisfying $|{\rm Im}(\lambda)|\in [A^{-1},A]$, we have as $h \to 0$
\begin{align}\label{eq:semiclassicalbounds}
\begin{split}
	\|(-hX + hV - \lambda)^{-1}\|_{\mc{H}_h^{N} \to \mc{H}_h^{N}} &= o(h^{-2}),\quad \,\,\,\,h^{-1}\lambda \in \mc{S}_0(\varepsilon), \quad N \geq \frac{1}{2},\\
	\|(-hX + hV - \lambda)^{-1}\|_{\mc{H}_h^{N} \to \mc{H}_h^{N}} &= \mc{O}(h^{-2}), \quad \,\,h^{-1}\lambda \in \mc{S}_1(\varepsilon), \quad N \geq 1.
\end{split}
\end{align}
\end{theo}
Here again the existence of a band structure for Ruelle resonances in the case of a potential is proved in \cite{FaTs1,FaTs2,FaTs3}.

\textbf{Acknowledgements.} This project has received funding from the European Research Council (ERC) under the European Union’s Horizon 2020 research and innovation programme (grant agreement No. 725967). We thank Semyon Dyatlov for several useful discussions and for suggesting to apply the method of \cite{Dy} in this setting.

\section{Anosov flows, anisotropic spaces and Ruelle resonances}
\subsection{Dynamical background}\label{sec:dynback}
Let $\mc{M}$ be a $3$-dimensional closed manifold equipped with a Riemannian metric $g$, which also induces a metric on $T^*\mc{M}$, the Sasaki metric.  
We will consider a smooth vector field $X$ on $\mc{M}$ and we assume that the flow $\varphi_t$ of $X$ is Anosov: 
there is a $d\varphi_t$ invariant splitting 
\[ T\mc{M}=E_0\oplus E_u\oplus E_s\]
where $E_0=\rr X$ is the flow direction, $E_s$ is the stable bundle defined by 
\[ v\in E_s \iff \exists C>0,\nu>0, \forall t\geq 0,\quad  \|d\varphi_t(v)\|\leq Ce^{-\nu t}\|v\| \]
and $E_u$ is the unstable bundle defined by 
\[ v\in E_u \iff \exists C>0,\nu>0, \forall t\leq 0,\quad  \|d\varphi_t(v)\|\leq Ce^{-\nu |t|}\|v\|. \]
The bundles $E_u$ and $E_s$ are $1$-dimensional vector bundles, which have H\"older regularity.
We say that the flow is a \emph{contact} Anosov flow if there is a smooth $1$-form $\alpha$
such that $\alpha(X)=1$, $\ker d\alpha=E_u\oplus E_s$ and $d\alpha$ is non-degenerate on $\ker \alpha$, or equivalently $\alpha\wedge d\alpha$ is a smooth volume form on $\mc{M}$. 
If the flow of $X$ is a contact Anosov flow, then the stable and unstable bundles $E_s, E_u$ have regularity $C^{2-\eps}(\mc{M})$ for all $\eps>0 $ by the result of Hurder-Katok \cite{HuKa}. For notational convenience we will denote $C^{k-}(\mc{M}):=\cap_{\eps>0}C^{k-\eps}(\mc{M})$ for $k\in\nn$. 
We shall assume that $E_s$ and $E_u$ are orientable, which means that there are two non-vanishing vector fields $U_\pm\in C^{2-}(\mc{M};T\mc{M})$ such that 
\[ E_s= \rr U_+, \quad E_u= \rr U_-.\] 
Note that the orientability condition for $E_u,E_s$ is satisfied for Anosov geodesic flows on orientable surfaces.
We can define a dual Anosov decomposition of $T^*\mc{M}$ 
\[ T^*\mc{M}=E_0^*\oplus E_u^*\oplus E_s^*, \quad \textrm{with } E_u^*(E_u\oplus \rr X)=0, \,\, E_s^*(E_s\oplus\rr X)=0\]
and $E_0^*=\rr \alpha$ is the annihilator of $E_u\oplus E_s$.
We will also define 
\begin{equation}\label{Gammapm}
\Gamma_+ := E_u^*\oplus E_0^* , \quad \Gamma_-=E_s^*\oplus E_0^*, \quad K=\Gamma_+\cap\Gamma_-=E_0^*
\end{equation}
that we call the outgoing tail, the incoming tail and the trapped set for $X$, respectively.

Let us now recall a couple of properties about $U_\pm$ from \cite[Lemma 2.2]{FaGu}: we have the following commutation formula
\begin{equation}\label{eq:commute0}
	[X, U_\pm] = \pm r_\pm U_\pm,
\end{equation}
and $r_\pm, U_\pm \in C^{2-}$. From \eqref{eq:commute0}, we conclude for all $t \in \mathbb{R}$
\begin{align}\label{eq:Upm}
\begin{split}
	d\varphi_{-t}(x) U_-(x) &= e^{-\int_{-t}^0 \varphi_s^*r_-(x) ds} U_-(\varphi_{-t}(x)),\\
	d\varphi_{t}(x) U_+(x) &= e^{-\int_{0}^t \varphi_s^*r_+(x) ds} U_+(\varphi_t(x)).
\end{split}
\end{align}
We note that both $U_-$ and $r_-$ are not uniquely defined but the large time average of $r_-$ along orbits are intrinsic to $X$. By Fekete's lemma, we may define the following finite, positive quantities
\begin{align}\label{eq:muminmax}
	\nu_{\min} := \lim_{t \to \infty} \inf_{x \in \mc{M}} \frac{1}{t}\int_0^t r_-(\varphi_{-s}(x))ds,\quad 
	\nu_{\max} := \lim_{t \to \infty} \sup_{x \in \mc{M}} \frac{1}{t}\int_0^t r_-(\varphi_{-s}(x))ds.
\end{align}
Note that in the contact case, one also has 
\begin{equation}\label{eq:muminmax2}
\nu_{\min} = \lim_{t \to \infty} \inf_{x \in \mc{M}} \frac{1}{t}\int_0^t r_+(\varphi_{s}(x))ds,\quad 
	\nu_{\max} = \lim_{t \to \infty} \sup_{x \in \mc{M}} \frac{1}{t}\int_0^t r_+(\varphi_{s}(x))ds.
\end{equation}
Then for each $\varepsilon > 0$, there exists $C_\varepsilon > 0$ such that for all $t \geq 0$ and $x\in \mc{M}$
\[C_\varepsilon^{-1} e^{-(\nu_{\max} + \varepsilon)t} \leq |d\varphi_{\pm t}U_\pm (x)| \leq C_\varepsilon e^{-(\nu_{\min} - \varepsilon)t}.\]
 
\subsection{Anisotropic space and extension of the resolvent of $X$}

Let $V\in C^\infty(\mc{M})$ be a potential. We want to prove a spectral gap with resolvent estimates at high frequency for $-X+V$. 
It is convenient to make the semiclassical rescaling  
\begin{equation}\label{eq:P_h}
	P_h:=-ihX+ihV, \quad P_h(\la)=P_h-i\la 
\end{equation}
where $h>0$ is a small parameter and $\la \in \cc$. 
The semiclassical principal symbol of $P_h$ is given by
\begin{equation}\label{Hamiltonian} 
p(x,\xi)=\xi(X).
\end{equation}
We will denote by $\Phi_t$ the Hamiltonian flow at time $t$ of $p$: notice that 
\[\Phi_t(x,\xi)=(\varphi_t(x),(d\varphi_t(x)^{-1})^T\xi). \]
The Hamiltonian vector field of $p$ will be denoted by $H_p$ so that $\Phi_t=e^{tH_p}$.
In \cite{FaSj}, Faure and Sj\"ostrand construct 
a family of Hilbert spaces, called anisotropic Sobolev spaces, using variable order pseudo-differential operators. Another presentation is given by Dyatlov-Zworski \cite{DyZw}, where a semiclassical parameter is included. We recall a few results about these spaces and the spectral properties of $P_h$ acting on them, we refer to \cite{FaRoSj,FaSj,DyZw} for more details.

If $m\in S^0(T^*\mc{M})$ is a symbol of degree $0$ on $T^*\mc{M}$, i.e. satisfying the bounds in local coordinates (for some constants $C_{\alpha,\beta}>0$)
\[ |\pl_x^{\alpha}\pl_\xi^\beta m(x,\xi)|\leq C_{\alpha,\beta}\cjg \xi\cjd^{-|\beta|},\]
we denote by $\Psi_h^m(\mc{M})$ the space of semiclassical pseudo-differential operators of order $m$ as defined in \cite[Appendix]{FaRoSj} (see also \cite{DyZw} for the $h$-dependent version): these are operators which have the form in local coordinates 
\[ Au(x)=(2\pi h)^{-3}\int e^{\frac{i}{h}(x-y)\xi}a(x,\xi)u(y)dyd\xi\]  
where $a\in S^m(T^*\mc{M})$ is a symbol of order $m$, ie. it satisfies local bounds for each $\eps>0$ small
\[ |\pl_x^{\alpha}\pl_\xi^\beta a(x,\xi)|\leq C_{\alpha,\beta,\eps}\cjg \xi\cjd^{m(x,\xi)-(1-\eps)|\beta|}.\] 
We also fix  a semiclassical quantization ${\rm Op}_h$ on $\mc{M}$ mapping symbols to operators acting on $L^2$ (see \cite{Zw} and \cite[Appendix E]{DyZw}). The space $\Psi_h^{\rm comp}(\mc{M})$ denotes the space of compactly microsupported operators, defined as the space of operators of the form ${\rm Op}_h(a)+h^\infty \Psi_h^{-\infty}(\mc{M})$ 
with $a\in C_c^\infty(T^*\mc{M})$. We shall denote by $\sigma: \Psi_h^m(\mc{M})\to S^m(T^*\mc{M})/hS^{m-1+\eps}(T^*\mc{M})$ the semiclassical principal symbol map, satisfying $\sigma({\rm Op}_h(a)) - a \in hS^{m-1+\eps}(T^*\mc{M})$ for all $\eps>0$. Finally, let $H_h^s:=(1+h^2\Delta_g)^{-s/2}L^2(\mc{M})$ be the semiclassical Sobolev space of order $s$ (here $\Delta_g$ is the Riemannian Laplacian of $g$ on $\mc{M}$).

First, by \cite{FaSj} there exists an \emph{escape function} $G\in C^\infty(T^*\mc{M})$
of the form 
\[G(x,\xi)=m(x,\xi)\log f(x,\xi)\]
where $m\in C^\infty(T^*\mc{M};[-1,1])$ is homogeneous of degree $0$ in the fiber variable $\xi$ for $|\xi|>1$, $f>0$ is homogeneous of degree $1$ in the fiber variable for $|\xi|>1$ and satisfies the following properties: there is a $C_m>0$ such that  
\[ \left\{\begin{array}{l}
m(x,\xi)=-1 \textrm{ in a conical neighborhood }V_u \textrm{ of }E_u^*,\\
m(x,\xi)=1 \textrm{ in a conical neighborhood }V_s \textrm{ of }E_s^*, \\
H_pm\leq 0 \textrm{ on }T^*\mc{M}, \\
H_pG<-C_m \textrm{ outside a conical neighborhood } V_0 \textrm{ of }E_0^*.
\end{array}\right.\] 
We can thus define the anisotropic Sobolev space of order $N\in \rr$ to be, for $h$ small enough
\[\mc{H}_h^{NG}:={\rm Op}_h(e^{NG})^{-1}L^2(\mc{M}).\]
Here ${\rm Op}_h(e^{NG})$ belongs to the class of semiclassical pseudodifferential operators $\Psi_h^{Nm}(\mc{M})$ of variable order $m$, and it is invertible in the class of semiclassical pseudodifferential operators with inverse ${\rm Op}_h(e^{NG})^{-1}\in \Psi_h^{-Nm}(\mc{M})$ if $h>0$ is small enough. We notice that, since 
${\rm Op}_h(\cjg \xi\cjd^{-N}){\rm Op}_h(e^{NG})\in \Psi_h^0(\mc{M})$, there is $C>0$ such that for all $u\in C^\infty(\mc{M})$ and $h>0$
\begin{equation}\label{eq:isotropic-anisotropic}
	\frac{1}{C}\|u\|_{H_h^{-N}}\leq \|u\|_{\mc{H}_h^{NG}}\leq C\|u\|_{H_h^N}.
\end{equation}

The following result was proved in \cite{FaSj} (see also \cite{DyZw})
\begin{prop}\label{Faure-Sjostrand}
There is\footnote{It can be checked, by inspecting the proof, that one can choose $c_{X}=\nu_{\min}$ and $c_V=V_{\max}$.} $c_{X}>0$, $c_V\in\rr$ such that the operator $P_h(h\la):\mc{D}(P_h)\subset \mc{H}_h^{NG}\to \mc{H}_h^{NG}$ is Fredholm in the region ${\rm Re}(\la)>c_V-c_{X}N$, where $\mc{D}(P_h)=\{u\in \mc{H}_h^{NG}\ |\ P_hu\in \mc{H}_h^{NG}\}$. Its inverse $(P_h-i\la h)^{-1}$, 
called the resolvent of $X$, is a meromorphic family of bounded operators on $\mc{H}_h^{NG}$ and the poles are called Ruelle resonances.
\end{prop} 

\section{Semiclassical measures associated to sequences of resonances and quasimodes}

In this Section, we shall see that the presence of an infinite number of
Ruelle resonances $\la_n$ with ${\rm Re}(\la_n)\to -\gamma$ and $|\la_n|\to \infty$ implies the existence of a non-trivial measure $T^*\mc{M}$, with certain invariance properties with respect to $H_p$, and more generally the same holds for good quasimodes. This will allow us to prove gaps of resonances and bounds on the resolvent. The idea is to rescale the equation $(-X + V -\la_n)u_n=r_n$ with $r_n$ small in some appropriate norm by setting $h_n=1/|\la_n|$ so that, according to \eqref{eq:P_h}, $(P_{h_n}-ih_n\la_n)u_n=ih_nr_n$ and to view this last equation as a semiclassical problem.
This method follows some ideas developed by Dyatlov in \cite{Dy} for operators with $r$-normally hyperbolic trapped set. 
\begin{equation} \label{eq:A1}
	    \begin{minipage}{0.85\textwidth}
	        {\bf Assumption 1.} For some $\gamma > 0$ and $N > 0$, consider a sequence 
\begin{equation*}\label{def:la_h}
\Re \la_h=-h\gamma+o(h), \quad \Im \la_h = 1 + o(1), \quad -\gamma > V_{\max} - N\nu_{\min},
\end{equation*} 
where $h>0$ is a parameter going to $0$, which can be discrete or continuous. 
Let $u_h\in \mc{D}(P_h)\subset \mc{H}_h^{NG}$ be a sequence (as $h\to 0$) of quasimodes for $P_h(\la_h)$:
\begin{equation*}\label{hypuh}
\|u_h\|_{\mc{H}_h^{NG}}=1,\quad \|P_h(\la_h)u_h\|_{\mc{H}_h^{NG}}=o(h^{\beta}),
\end{equation*}
for some $\beta> 0$.
	    \end{minipage}\tag{A1}
	\end{equation}
 
\noindent First, by applying \cite[Theorem E.42]{DyZw}, we directly have the 
\begin{lemm}\label{lemm:semiclassicalmeasureexistence}
Under the assumptions of \eqref{eq:A1}, there is a subsequence $u_{h_j}$ and a Radon measure measure $\mu$ on $T^*\mc{M}$ such that for each $A_h\in \Psi_h^{{\rm comp}}(\mc{M})$
\[ \cjg A_{h_j}u_{h_j},u_{h_j}\cjd_{L^2}\to \int_{T^*\mc{M}}\sigma(A) d\mu \textrm{ as }j\to \infty .\]
\end{lemm}
We can also consider the sequence $\til{u}_{h_j}={\rm Op}_{h_j}(e^{NG})u_{h_j}$ which satisfies $\|\til{u}_h\|_{L^2}=1$
and thus one can consider its  semiclassical measure $\til{\mu}$ satisfying, as $j\to \infty$,
\[ \cjg A_{h_j}\til{u}_{h_j},\til{u}_{h_j}\cjd_{L^2}\to \int_{T^*\mc{M}}\sigma(A) d\til{\mu}.\]
\subsection{Support and first invariance properties of the semiclassical measure}

We first remark that the measure $\mu$ must satisfy some support properties and some invariance under the Hamiltonian flow of the $p(x,\xi)$ (defined in \eqref{Hamiltonian}). These are 
consequences of elliptic regularity and propagation of singularities.
\begin{lemm}\label{lemm:semiclassicalmeasureproperties}
The measures $\til{\mu}$ and $\mu$ satisfy $e^{2NG}d\mu=d\til{\mu}$, and 
\[\supp(\mu)\subset \{ (x,\xi)\in T^*\mc{M}\ |\ \xi(X)=-1\}.\]
Assume that $\beta\geq 1$ in \eqref{eq:A1}. If $p(x,\xi)=\xi(X)$ and $H_{p}$ is the Hamiltonian vector field of $p$, then for each $a\in C^\infty_c(T^*\mc{M})$
\begin{equation}\label{dampingeq}
 \int_{T^*\mc{M}}(H_{p}+2(\gamma+V))a \,d\mu=0.
 \end{equation}
\end{lemm}
\begin{proof} 
The first identity follows from the definition of $\mu,\til{\mu}$ for $A\in \Psi_h^{\rm comp}(\mc{M})$ 
\[\begin{split}
\cjg A_{h_j}\til{u}_{h_j},\til{u}_{h_j}\cjd_{L^2}&= \cjg {\rm Op}_{h_j}(e^{NG})^*A{\rm Op}_{h_j}(e^{NG})u_j,u_j\cjd_{L^2}\\
 & \to_{j\to \infty} \int_{T^*\mc{M}}\sigma({\rm Op}_{h_j}(e^{NG})^*A{\rm Op}_{h_j}(e^{NG}))d\mu=\int_{T^*\mc{M}}\sigma(A)e^{2NG}d\mu.
\end{split}\]
The semiclassical principal symbol of $P_h(\la_h)$ is $\xi(X) + 1 +o(1)$ thus 
\[\supp(\mu)\subset \{ (x,\xi)\in T^*\mc{M}\ | \ \xi(X)=-1\}\]
using microlocal ellipticity \cite[Theorem E.43]{DyZw}. 
We write ${\rm Im}(P_h(\la_h))=(P_h(\la_h)-P_h(\la_h)^*)/2i\in h\Psi^0_h(\mc{M})$, with 
\[ h^{-1}{\rm Im}(P_h(\la_h)) = \gamma + V + o(1).\]
By \cite[Theorem E.44]{DyZw} (or Proposition \ref{prop:semiclassicalmeasurepropagation}), we have that for each $a\in C_c^\infty(T^*\mc{M})$,
\eqref{dampingeq} holds.
\end{proof}

Next, we shall use propagation estimates to obtain information on the support of $\mu$ and establish that $\mu\not =0$.
\begin{lemm}\label{lemm:support}
Assume $\beta\geq 1$ in \eqref{eq:A1}, then the semiclassical measure $\mu$ satisfies 
\begin{equation}\label{supp2} 
\supp(\mu)\subset \Gamma_+ = E_0^*\oplus E_u^*.
\end{equation}
Moreover, if $K_{R}:= \{(x,\xi)\in T^*\mc{M}\,|\, d_g((x,\xi),E_0^*) \leq R\}$, 
we have $\mu(K_R)>0$ for any $R>0$. 
\end{lemm}
\begin{proof} We fix $N_1>0$ an arbitrarily large constant. 
Let us consider the fiber compactification $\bbar{T}^*\mc{M}$ of $T^*\mc{M}$ (see \cite[Appendix E]{DyZw} for the definition). The bundles $E_u^*,E_s^*$ and $E_0^*$ extend to $\bbar{T}^*\mc{M}$ naturally (by taking their closures) and we let 
$L:=E_s^*\cap \pl\bbar{T}^*\mc{M}$ and $L':=E_u^*\cap \pl\bbar{T}^*\mc{M}$. The set $L$ is a radial source and $L'$ a radial sink in the terminology of \cite[Definition E.50]{DyZw}.
First, we will show that $\mu$ has no mass for $\xi$ large enough in a conic neighborhood $V_s$ of $E_s^*$. This can be proved using the radial point estimates of \cite[Theorem E.52]{DyZw}.  There is $s_0\in\rr$ such that for all $s>s_0$ and $N_1>0$, and $B_1\in \Psi^0_h(\mc{M})$ with $L\subset {\rm ell}_h(B_1)$, 
there is $A_0\in \Psi^0_h(\mc{M})$ elliptic in a neighborhood $U$ of $L$ in $\bbar{T}^*\mc{M}$, and $C>0$ so that for all $u\in H_h^s(\mc{M})$ such that $B_1P_h(\la_h)u\in H_h^s$
\begin{equation}\label{eq:radialsource}
	 \|A_0u\|_{H_h^s}\leq Ch^{-1}\|B_1P_h(\la_h)u\|_{H_h^{s}}+Ch^{N_1}\|u\|_{H_h^{-N_1}}.
\end{equation}
Note that by \eqref{eq:A1} we may choose $s_0 = N - \delta_0$ for some $\delta_0 > 0$. Now, since $u_h\in \mc{H}_h^{NG}$ and $P_h(\la_h)u_h\in \mc{H}_h^{NG}$, we can choose $B_1$ so that ${\rm WF}_h(B_1)\subset V_s$ and, since $B_1{\rm Op}_h(e^{NG})^{-1}\in \Psi_h^{-N}(\mc{M})$ by the fact that $m=1$ in $V_s$, we see that $B_1P_h(\la_h)u_h\in H_h^N$ and $\|B_1P_h(\la_h)u_h\|_{H_h^N}=o(h^{\beta})$. We thus choose $s=N$ with $N>s_0$ and get that there is $A_0\in \Psi_h^0(\mc{M})$ elliptic in a neighborhood $U$ of $L$ such that 
\begin{equation}\label{A''uh}
\|A_0u_h\|_{H_h^N}=o(h^{\beta-1}).
\end{equation}
This shows that $\mu(\chi|\sigma(A_0)|^2)=0$ for all $\chi\in C_c^\infty(T^*\mc{M})$ 
and thus $\mu$ vanishes in $U$. By the Anosov property of $X$ ($L$ is an attractor for the backward Hamiltonian flow $\Phi_{-t}=e^{-tH_p}$) for all 
$(x,\xi)\in T^*\mc{M}\setminus E_u^*\oplus E_0^*$, there is $T>0$ such that $\Phi_{-T}(x,\xi)\in U$. We can then use the invariance \eqref{dampingeq} to deduce that $\mu=0$ outside $E_u^*\oplus E_0^*$ so that \eqref{supp2} holds. 

To prove that $\mu\not=0$, we also need to have $H_h^{-N}$ estimates on 
$Qu_h$ for some appropriate $Q\in \Psi_h^{0}(\mc{M})$ microsupported outside $E_u^*$. 

For an operator $Y\in \Psi_h^m(\mc{M})$, denote by $Y^{(N)} := \Op_h(e^{NG}) Y \Op_h(e^{NG})^{-1}$ the conjugated operator. By \cite[Appendix]{FaRoSj} we then have $Y^{(N)} \in \Psi_h^m(\mc{M})$ and $\sigma_h(Y^{(N)}) = \sigma_h(Y)$. First, the elliptic 
estimate of \cite[Theorem E.33]{DyZw} gives that for each $A_1\in \Psi^0_h(\mc{M})$ with ${\rm WF}_h(A_1)\cap \{p=-1\}=\emptyset$  and each $N_1>0$, there is $C>0$ so that for all $h>0$
\begin{equation}\label{Q0u0}  
\|A_1^{(N)} \til{u}_h\|_{L^2}\leq C\|\big(P_h(\la_h)\big)^{(N)}\til{u}_h\|_{L^2} + Ch^{N_1}\|u_h\|_{H_h^{-N}}.
\end{equation}
Here we used the above facts for $Y = A_1, P_h(\lambda_h)$. Therefore, for each cone $\mc{C}_{0}\subset \bbar{T}^*\mc{M}$ satisfying $\mc{C}_{0}\cap (E_u^*\oplus E_s^*)=\emptyset$, there is $R>0$ large and $A_1\in \Psi_h^0(\mc{M})$ satisfying $\sigma(A_1)= 1$ in $\mc{C}_{0}\cap\{|\xi|>R\}$ such that, for some $N_1 > \beta$
\begin{equation}\label{Q0u}  
\|A_1 u_h\|_{\mc{H}_h^{NG}}\leq C\|P_h(\la_h) u_h\|_{\mc{H}_h^{NG}} + Ch^{N_1}\|u_h\|_{H_h^{-N}}=o(h^{\beta}). 
\end{equation}

Next, we remark from the Anosov property of $X$ that for each cone $\mc{C}_s \subset \bbar{T}^*\mc{M}\setminus (E_u^*\oplus E_0^*)$ there is $T>0$ so that 
$\Phi_{-T}(\mc{C}_s)\subset U$ where $U$ is the neighborhood of $L$ used before, we can then use the propagation of singularity estimate of \cite[Theorem E.47]{DyZw} to deduce the following. For each cone $\mc{C}_s$ as above, there is $A_2\in \Psi_h^0(\mc{M})$ with $\sigma(A_2)=1$ on $\mc{C}_s$, $B\in\Psi_h^0(\mc{M})$ with ${\rm WF}_h(B)\subset U$, so that for all $N_1>0$, there is  a constant $C>0$ such that
\[ \|A_2u_h\|_{\mc{H}_h^{NG}}\leq Ch^{-1}\|P_h(\la_h)u_h\|_{\mc{H}_h^{NG}}+C\|Bu_h\|_{\mc{H}_h^{NG}}+Ch^{N_1}\|u_h\|_{H_h^{-N}}.\]
Note that, strictly speaking, we applied the propagation of singularities estimate to $(A_2)^{(N)}$, $\big(P_h(\lambda_h)\big)^{(N)}$, $B^{(N)}$ and $\til{u}_h$, similarly to \eqref{Q0u0}. Using \eqref{A''uh} and elliptic estimates, $\|Bu_h\|_{\mc{H}_h^{NG}}\leq C\|A_0u_h\|_{H_h^{N}} + \mc{O}(h^{\infty}) = o(h^{\beta-1})$, and combining with \eqref{eq:A1}, we deduce that 
\begin{equation}\label{Q1u} 
\|A_2u_h\|_{\mc{H}_h^{NG}}=o(h^{\beta-1}).
\end{equation}

Next, by \cite[Theorem E.54]{DyZw}, there is $s_1$ such that for all $s<s_1$, $N_1 > 0$, there is $A_3\in \Psi^0_h(\mc{M})$ elliptic near $L':=E_u^*\cap \pl\bbar{T}^*\mc{M}$, $B'\in \Psi^0_h(\mc{M})$ with ${\rm WF}_h(B')\cap L'=\emptyset$, $B_1\in \Psi_h^0(\mc{M})$ elliptic on ${\rm WF}_h(A_3)$ and $C>0$ such that for all $v\in H^{s}_h(\mc{M})$ such that $B_1P_h(\la_h)v\in H_h^s$, all $h>0$ 
\begin{equation}\label{radialsink} 
\|A_3v\|_{H_h^s}\leq C\|B'v\|_{H^s_h}+Ch^{-1}\|B_1P_h(\la_h)v\|_{H_h^s}+Ch^{N_1}\|v\|_{H_h^{-N}}.
\end{equation}
Note that, by possibly applying another time the propagation of singularities estimate, 
one can choose $B'$ to be microsupported also close to $L'$; we can assume that $A_3,B'$ are microsupported where $C_1\cjg \xi\cjd^{-N}\leq e^{NG}\leq C_2\cjg \xi\cjd^{-N}$. Note that by \eqref{eq:A1}, we may choose $s_1 = -N + \delta_1$ for some $\delta_1 > 0$. Let $\til{A}_3\in \Psi_h^0(\mc{M})$ elliptic so that ${\rm WF}_h(\til{A}_3-A_3)\cap L'=\emptyset$, i.e. $\widetilde{A}_3$ is microlocally equal to $A_3$ near $L'$. We can now write for some $N_1 > \beta$ and some $C,C'>0$, using \eqref{radialsink} with $v = u_h$ and $s = -N$ in the third line
\begin{equation}\label{bornesupuh}
\begin{split} 
1=\|u_h\|_{\mc{H}_h^{NG}}\leq & C\|\widetilde{A}_3u_h\|_{\mc{H}_h^{NG}} +\mc{O}(h^{N_1})\\
\leq & C\|A_3u_h\|_{H_h^{-N}}+C\|(\widetilde{A}_3-A_3)u_h\|_{\mc{H}_h^{NG}}+\mc{O}(h^{N_1})\\
\leq & C'\|B'u_h\|_{H^{-N}_h}+C\|(\widetilde{A}_3-A_3)u_h\|_{\mc{H}_h^{NG}}+o(h^{\beta-1}).
\end{split}
\end{equation}

Note that $B'$ and $(\widetilde{A}_3-A_3)$ satisfy the same property that their wavefront set 
does not intersect $L'$ in $\bbar{T}^*\mc{M}$. But if $B''\in \Psi_h^0(\mc{M})$ is such that ${\rm WF}_h(B'')\cap L'=\emptyset$, then one can choose $\mc{C}_s$ and $\mc{C}_0$ above so that 
${\rm WF}_h(B'')\subset \mc{C}_s\cup \mc{C}_0\cup \{|\xi|\leq 2R\}$ for some $R>1$ large enough, and thus there exists $Q\in \Psi_h^{\rm comp}(\mc{M})$ such that, for any $N_1 > 0$
\[
	 \|B'' u_h\|_{\mc{H}^{NG}_h}\leq C(\|B'' A_1 u_h\|_{\mc{H}^{NG}_h}+\|B''A_2 u_h\|_{\mc{H}^{NG}_h}+
\|B''Qu_h\|_{H^{NG}_h})+C h^{N_1} \|u_h\|_{H^{-N}_h}.
\]

Applying this to $B''=B'$ and $B''=\widetilde{A}_3-A_3$ and using \eqref{Q1u}, \eqref{Q0u} and \eqref{bornesupuh}, we deduce that there is $Q\in \Psi_h^{\rm comp}(\mc{M})$ elliptic on $\{|\xi|\leq R\}$ and $C>0$ such that 
\[ 1\leq C\|Qu_h\|_{\mc{H}^{NG}_h}+\mc{O}(h^{N_1})+o(h^{\beta-1}).\]
Taking some $N_1 > \beta$ and since $\beta\geq 1$, the right hand side converges to $0$ as $h\to 0$ and we deduce that 
$\mu(|\sigma(Q)|^2)>0$, showing that $\mu(\{|\xi|\leq R\})>0$ for $R>0$ large enough. Using the flow invariance \eqref{dampingeq} and the fact that $e^{tH_{p}}$ is expanding linearly in $\xi$ in the $E_u^*$ direction where $\mu$ is supported, we deduce that $\mu(K_{R}) > 0$ for all $R>0$.
\end{proof}
We remark that the proof of this Lemma shows that $u=o(h^{\beta-1})$ microlocally on compact sets 
outside $\Gamma_+\cap \{p=-1\}$ (use \eqref{Q0u0} and \eqref{Q1u}).

\section{Unstable derivatives and exotic calculus}

If $V\in C^\infty(\mc{M})$ is a potential, the proof of \cite[Theorem 2]{FaGu} gives that
there is $\alpha_V\in C^{\frac{\nu_{\min}}{\nu_{\max}}-}(\mc{M})$ such that $(-X-r_-)\alpha_V=U_-(V)$ and
\begin{equation}\label{[X+V,U_-]}
[-X+V,U_-+\alpha_V]=r_-(U_-+\alpha_V), \quad r_-\in C^{2-}(\mc{M}), \quad X\alpha_V\in C^{\frac{\nu_{\min}}{\nu_{\max}}-}(\mc{M}).
\end{equation}
Recall that $u$ is called a \emph{generalized resonant state} at $s$ if $(X + s)^ku = 0$ for some $k \geq 0$ and $u  \in \mc{H}^{N}$ for some $N$ large enough. We recall the result of \cite{FaGu}:
\begin{prop}[Faure-Guillarmou \cite{FaGu}]
If $X$ is a contact 3-dimensional Anosov flow, the Ruelle generalized resonant states $u$ with resonance in ${\rm Re}(s)>-\nu_{\min}$ satisfy $U_-u=0$. 
\end{prop} 
This extra property of resonant states will give us a new identity on the semiclassical measures associated to sequences of resonances: roughly speaking $\mu$ will be invariant by the Hamiltonian flow of $\sigma(U_-)$. More generally we can use this unstable derivative even for quasimodes. However, there is a technical drawback, which is that $U_-$ is not smooth. For microlocal methods, this complicates the argument, and we will have to regularize $U_-$ with an $h$-dependent scale. 
This leads us to use a slightly exotic class of pseudo-differential operators and symbols. 

\subsection{Regularization and exotic pseudo-differential calculus}

If $m\in \rr$, $\rho\in (0,1)$ and $k\in\rr^+$, we define the exotic pseudo-differential calculus $\Psi^m_{h,\rho,k}(\mc{M})$ to be the set of operators of the form ${\rm Op}_h(a)$ where the symbol $a$ in the class $S^m_{h,\rho,k}(T^*\mc{M})$, which is the space of smooth functions on $T^*\mc{M}$ such that in local coordinates (here $x_+:=\max(x,0)$ for $x\in \rr$)
\begin{equation}\label{eq:exoticsymbolclass}
\forall \alpha,\beta, \quad	|\pl_x^{\alpha}\pl_\xi^{\beta}a(x,\xi)|\leq C_{\alpha,\beta,\eps}\cjg \xi\cjd^{m-|\beta|} h^{-\rho (|\alpha| -k)_+}. 
	\end{equation}
We will also write $\Psi^m_{h,\rho}
(\mc{M})$ when $k=0$.
The calculus defined by this property has all the good properties of a semiclassical pseudo-differential calculus, and we refer to Appendix \ref{app:exoticcalc} for a summary of the needed properties that we shall use freely in this section.  

 We want to work with smooth coefficients and will thus regularize H\"older functions at scale $h^\rho$ as follows. We fix a partition of unity $(\psi_j)_j$ associated to a family of local charts $(\mc{O}_j)_j$ (i.e. $\supp(\psi_j)\subset \mc{O}_j$), and denote by $(x_1,x_2,x_3): \mc{O}_j\to \rr^3$ the local coordinates in each chart. Without loss of generality, we shall choose the coordinate systems in the charts so that $X=\pl_{x_1}$. 
Then if $u	\in C^{k}(\mc{M})$ for $k\geq 0$, define $u^h:=\sum_j\psi_ju_j^h$ where $u_j=u|_{\mc{O}_j}$ and the $h^\rho$ regularization in the chart $\mc{O}_j$, identified to an open set of $\rr^3$, is given by
\begin{equation}\label{h-reg} 
\forall a\in L_{\rm comp}^2(\rr^3),\quad  a^h(x):= h^{-3\rho}\int_{\rr^3}\chi((x-x')/h^{\rho})a(x')dx'
\end{equation}
where $\chi\in C_c^\infty(\rr^3;\rr^+)$ has integral $1$ and is chosen to be a radial function, i.e. a function of $|x|$. Note that $\psi_j u_j^h$ depends only on the value of $u$ on an $h^\rho$-neighborhood of ${\rm supp}(\psi_j)$.
Notice also in that $1^h=1$. Similarly for a semiclassical differential operator $P$ with $C^{k}(\mc{M})$ coefficients, that is an operator which in each chart $\mc{O}_j$ has the form 
\[\sum_{\alpha} a_{j, \alpha}(x)h^{|\alpha|}\pl_{x}^{\alpha}\]
with $a_{j, \alpha} \in C^{k}(\mc{O}_j)$, we define $P^h=\sum_{j} \psi_j P_j^h$ where $P_j=P|_{\mc{O}_j}$, the regularization in each chart $\mc{O}_j$, is given by
\[ P_j^h=\sum_{\alpha} a_{j, \alpha}^h(x)h^{|\alpha|}\pl_{x}^{\alpha}.\]  
Moreover, if $P$ is a differential operator (i.e. non-semiclassical), we define $P^h$ analogously as an $h$-dependent differential operator. For instance, for a vector field $Y$ we then have $(hY)^h = hY^h$. We notice that this process depends a priori on the charts and a partition of unity, but this will not cause us any problems.
The following holds true:
 
\begin{lemm}\label{l:approximation}
Let $k\in (0,2)$ and let $a\in C^{k}(\mc{M})$. For $m\in [0,k)$ integer and $\ell\in \nn$, one has  
\[\begin{gathered} 
\|a^h-a\|_{C^m}=\mc{O}(h^{\rho (k-m)}||a||_{C^{k}}), \quad \|a^h\|_{C^\ell}=\mc{O}(h^{-\rho (\ell-k)_+}),\\
\forall j,  \quad\|\psi_j a_j^h-\psi_j a^h\|_{C^m}=\mc{O}(h^{(k-m)\rho}).
 \end{gathered}\]
In particular one has $a^h\in S^{0}_{h,\rho,k}(\mc{M})$ and $\psi_ja_j^h=\psi_ja^h+\mc{O}_{S^0_{h,\rho,0}}(h^{\rho k})$. If $k<1$ but $Xa\in C^k(\mc{M})$, one also has
\[ Xa^h= Xa+\mc{O}_{C^0}(h^{k\rho}).\] 
If $P$ is a semiclassical differential operator of order $\ell$ on $\mc{M}$ with $C^{k}(\mc{M})$ coefficients, then $P^h\in \Psi^{\ell}_{h,\rho,k}(\mc{M})$ and 
\begin{equation}\label{P-P_h} 
P^h=P+\mc{O}_{C^0}(h^{\rho k}).
\end{equation} 
where the norm is the $C^0$ norm of the coefficients in the local bases $\pl_{x_i}$.
\end{lemm}
\begin{proof}
We work out the case $k\geq 1$, the case $k<1$ is similar and indeed simpler. We write in the 
chart $\mc{O}_j$ (on ${\rm supp}(\psi_j)$)
\[a_j^h(x)= \int_{\rr^3}\chi(y)a_j(x+h^\rho y)dy\]
and use that $a_j(x+h^\rho y)=a_j(x)+h^\rho da_j(x).y+\mc{O}_{C^0}(h^{k\rho}
\|a\|_{C^{k}})$. But since $\chi$ is radial, one has $\int \chi(y)da_j(x).y \,dy=0$, and this shows that 
\[ a_j^h(x)=a_j(x)+\mc{O}(h^{k\rho}||a||_{C^k}).\]
One has a similar estimate for the derivative  
\[ \pl_{x_i} a_j^h=\pl_{x_i} a_j+\mc{O}_{C^0}(h^{(k-1)\rho} \|a\|_{C^k}).\]
In general, for all $|\alpha|\geq 2$ we get 
$\pl_{x}^\alpha a_j^h=\mc{O}_{C^0}(h^{-\rho (|\alpha|-k)_+})$ since
\[\begin{split}
\pl_x^\alpha a_j^h(x)=&   (-1)^{|\alpha|} h^{-|\alpha|\rho}\int_{\rr^3}\pl_{y}^\alpha \chi(y)a_j(x+h^\rho y)dy\\ 
=&  (-1)^{|\alpha| - 1} h^{(1-|\alpha|)\rho}\int_{\rr^3}\pl_{y}^{\alpha'} \chi(y)\pl_{x_i}a_j(x+h^\rho y)dy
\end{split}\]
for some $\alpha'$ satisfying $|\alpha'|=|\alpha|-1$, with $i$ so that $\alpha_i>0$, and using that $\pl_{x_i}a_j(x+h^\rho y)=\pl_{x_i}a_j(x)+\mc{O}((h^\rho |y|)^{k-1})$ and $\int \pl_y^{\alpha'}\chi(y)dy=0$. Thus, if $V\in C^\infty(\mc{M};T\mc{M})$, one has 
\[ \begin{split}
Va^h & =\sum_{j}\big(V(\psi_j)a_j^h+\psi_j V(a_j^h)\big) = \sum_{j}\big(V(\psi_j)(a_j^h - a_j) + \psi_j V(a_j^h - a_j) + V(\psi_j) a_j +  \psi_j V(a_j)\big)\\
&= \sum_{j}\big(V(\psi_j)a+\psi_j V(a_j)\big) + \mc{O}_{C^0}(h^{(k-1)\rho})=Va+\mc{O}_{C^0}(h^{(k-1)\rho}).
\end{split}\]
Next, we also get from this analysis that for $m\leq k$, in $\mc{O}_j$
\[\pl_{x_i}^{m}(\psi_j(a_j^h-a^h))=\psi_j\pl_{x_i}^m(a_j-a)+\pl_{x_i}^m(\psi_j)(a_j-a)+
\mc{O}_{C^0}(h^{(k-m)\rho})=\mc{O}_{C^0}(h^{(k-m)\rho}).\] 
Next, if $a\in C^k$ with $Xa\in C^k$, by using that $X=\pl_{x_1}$ in each chart $\mc{O}_j$ so that $X(a_j^h)=(Xa_j)^h$, we can write as above
\begin{align*}
	Xa^h=\sum_{j}\big(X(\psi_j)a_j^h+\psi_jX(a_j)^h\big) &= \sum_{j}\big(X(\psi_j)(a_j^h-a_j) + \psi_j(Xa_j)^h\big) = (Xa)^h+\mc{O}_{C^0}(h^{k\rho}),
\end{align*}
where we used that $\sum_{j}\psi_j=1$. The analysis above also implies that for a semi-classical differential operator $P$ the 
 full local symbol of $P^h$ in charts satisfies 
\[|\pl_x^\alpha \pl_\xi^\beta\sigma_{\rm full}(P^h)(x,\xi)|\leq C_{\alpha,\beta}\cjg \xi\cjd^{\ell-|\beta|}h^{-\rho (|\alpha|-k)_+}
\] 
and that \eqref{P-P_h} holds.
\end{proof}
Applying the preceding lemma to $hU_-$, we get  for any $\varepsilon > 0$
\begin{equation}\label{regU_-} 
hU_-^h\in\Psi^{1}_{h,\rho, 2-\varepsilon}(\mc{M})
, \quad hU_-^h=hU_-+\mc{O}_{C^0}(h^{1+\rho(2-\eps)}||U_-||_{C^{2-\eps}}).
\end{equation}
For technical purposes, we need the following claim on the size of $(ab)^h - a^hb^h$:

\begin{lemm}\label{l:product}
Let $a\in C^{k}(\mc{M})$ and $b\in C^{\ell}(\mc{M})$ with $k, \ell\in (0,2)$. 
Then, if $a_j=a|_{\mc{O}_j}$ and $b_j=b|_{\mc{O}_j}$, one has 
\[ a^hb^h= (ab)^h+\mc{O}_{S^0_{h,\rho,0}}(h^{\rho \cdot \min(k, \ell)}), \quad \psi_j a_j^hb_j^h=\psi_j (a_jb_j)^h+\mc{O}_{S^0_{h,\rho,0}}(h^{\rho \cdot \min(k, \ell)}).\]
\end{lemm}
\begin{proof}
First we work in the chart $\mc{O}_j$. Using $b_j(x)=b_j^h(x)+\mc{O}(h^{\ell\rho})$ 
and $a_j(x)=a_j^h(x)+\mc{O}(h^{k\rho})$, the bound $|\pl_x^\alpha b_j^h|=\mc{O}(h^{-\rho (|\alpha|-\ell)_+})$ in ${\rm supp}(\psi_j)\subset \mc{O}_j$, and the Taylor expansion 
\[(a_j^hb_j^h)(x + h^\rho y) = (a_j^hb_j^h)(x) + d(a_j^h b_j^h)(x) \cdot h^\rho y + R(x, y),\]
where $R(\cdot, y) = \mc{O}_{S^0_{h, \rho, 0}}(h^{\rho \cdot \min (k, \ell)})$ uniformly in $y$, we obtain
\[\begin{split} 
(a_jb_j)^h(x)= &\int_{\rr^3} \chi(y)a_j(x+h^\rho y)b_j(x+h^\rho y)dy\\
 =&\int_{\rr^3} \chi(y)a_j^h(x+h^\rho y)b_j^h(x+h^\rho y)dy+h^{-3\rho}\int_{\mathbb{R}^3} \chi\big(\frac{x-x'}{h^\rho}\big)a_j(x')(b_j-b_j^h)(x')dx'\\
 & + h^{-3\rho}\int_{\mathbb{R}^3} \chi\big(\frac{x-x'}{h^\rho}\big)(a_j-a_j^h)(x')b_j^h(x')dx'\\
=&a_j^h(x)b_j^h(x)+h^\rho d(a_j^hb_j^h)(x) \cdot \underbrace{\int_{\rr^3} y\chi(y)\,dy}_{=0}+\mc{O}_{S^{0}_{h,\rho,0}}(h^{\rho \cdot \min(k,\ell)}).
\end{split}\]
This gives the local result  by using that $\chi$ is radial (so the $h^\rho$ factor vanishes). Now for the global result, by the first part of the lemma
\[ (ab)^h=\sum_j\psi_j a_j^hb_j^h+\mc{O}_{S^{0}_{h,\rho,0}}(h^{\rho \cdot \min(k, \ell)})=
a^hb^h+\sum_j\psi_j a_j^h(b_j^h-b^h)+\mc{O}_{S^{0}_{h,\rho,0}}(h^{\rho \cdot \min(k, \ell)})\]
which yields the result since $\psi_j (b_j^h-b^h)=\mc{O}_{S^0_{h,\rho,0}}(h^{\rho \ell})$ by Lemma \ref{l:approximation}.
\end{proof}
We next compute the regularized commutator of $h(X - V)$ and $h(U_- + \alpha_V)^h$ where $\alpha_V$ is the function appearing in \eqref{[X+V,U_-]}:
\begin{lemm}\label{[X,U_-^h]}
Assume $\beta_0 \in (0, \frac{\nu_{\min}}{\nu_{\max}})$ and $\rho > \frac{1}{2 - \beta_0}$. If $V\in C^\infty(\mc{M})$ is $h$-independent, the following commutation relation holds 
\[[h(X-V),h(U_-+\alpha_V)^h]=-r_-^hh^2(U_-+\alpha_V)^h +h^{2+\beta_0\rho}R_h,\]
where the upper index $h$ denotes the regularization at scale $h^\rho$ and $R_h\in \Psi_{h,\rho,0}^1(\mc{M})$.
\end{lemm}
\begin{proof} For $V\in C^\infty(\mc{M})$, let $U^V_-:=U_-+\alpha_V$, then  
\[ [h(X-V),h(U_-^V)^h]=
\sum_{j}\psi_j [h(X-V), h(U^V_{-,j})^h]+h^2X(\psi_j)(U^V_{-,j})^h\] 
where $U^V_{-,j}=U_-^V|_{\mc{O}_j}$. We start by computing $[h(X-V),h(U^V_{-,j})^h]$ in $\mc{O}_j$. 
Writing $X=\pl_{x_1}$ as before and $U_{-,j}=\sum_ia_i(x)\pl_{x_i}$ in $\mc{O}_{j}$ we have, using $\pl_{x_i}a^h=(\pl_{x_i}a)^h$, 
\[ [h(X-V),h(U_{-,j}^h+\alpha_{V,j}^h)]=h^2\Big(X(\alpha_{V,j}^h)+\sum_{i=1}^3\Big((Xa_i)^h\pl_{x_i}+a_i^h\pl_{x_i}(V) \Big)\Big).
\]
Recall by \eqref{[X+V,U_-]} that $X\alpha_V+U_-(V)=-r_-\alpha_V$ and that $\alpha_V,X\alpha_V\in C^{\beta_0}(\mc{M})$ for all $\beta_0\in (0,\nu_{\min}/\nu_{\max})$; we 
 can then use Lemma \ref{l:product} and Lemma \ref{l:approximation} to deduce that on ${\rm supp} (\psi_j)$, for any $\varepsilon > 0$
\[\begin{split}
[h(X-V),h(U_{-,j}^h+\alpha_{V,j}^h)]=& h^2 [X,U_-]_j^h+h^2(X\alpha_{V})_j^h+h^2(U_{-}(V))_j^h+\mc{O}_{\Psi^0_{h,\rho,0}}(h^{2+\rho(2-\eps)})\\
=& -h^2 r_{-,j}^h(U_{-,j}^h+\alpha_{V,j}^h)+\mc{O}_{\Psi^0_{h,\rho,0}}(h^{2+\rho \beta_0})+\mc{O}_{\Psi^1_{h,\rho,0}}(h^{1+\rho(2-\eps)}).
\end{split}\]
Now $(U^V_{-,j})^h=(U_{-,j})^h+\alpha_{V,j}^h$ with 
$h^2\psi_j(U_{-,j})^h\in h\Psi^1_{h,\rho,2-\eps}(\mc{M})$ and $h^2\psi_j\alpha_{V,j}^h\in h^2\Psi^0_{h,\rho,\beta_0}(\mc{M})$ satisfy, using again that $\sum_j \psi_j=1$ and Lemma \ref{l:approximation}, for any $\varepsilon > 0$
\begin{align*}
\sum_j hX(\psi_j)\alpha_{V,j}^h &= \sum_j hX(\psi_j)[\alpha_{V,j}^h-\alpha_V]=\mc{O}_{C^0}(h^{1+\beta_0\rho}),\\
\sum_j hX(\psi_j)(U_{-,j})^h &= \sum_j hX(\psi_j)[(U_{-,j})^h-U_-]
=\mc{O}_{C^\ell}(h^{1+\rho(2-\eps-\ell)}),
\end{align*}
for $\ell=0,1$, where $C^{\ell}$ denotes the norm on vector fields. In particular, $\sum_j hX(\psi_j)(U_{-,j})^h\in h^{\rho(2-\eps)}\Psi_{h,\rho,0}^1(\mc{M})$ and $\sum_j hX(\psi_j)\alpha_{V,j}^h\in h^{1+\beta_0\rho}\Psi^0_{h,\rho,0}(\mc{M})$. 
Thus the error term belongs to $h^{\min(2 + \rho \beta_0, 1 + \rho(2 - \varepsilon))} \Psi^1_{h, \rho, 0}(\mc{M})$ and using the assumption on $\rho$ completes the proof.
\end{proof}

\subsection{A propagation estimate and regularity of the semiclassical measure in the unstable direction}
In this section, we are going to use the regularized version of the commutation formula $[X,U_-]=-r_-U_-$ proved in Lemma \ref{[X,U_-^h]} to deduce that the semiclassical measure $\mu$ enjoys some extra regularity in the (lifted) unstable direction.  

Before proceeding, we define the semiclassical principal symbol of $-ihU_\pm$
\[ p_\pm (x,\xi): = \sigma(-ihU_\pm)(x, \xi) = \xi(U_\pm(x)).\]
Recall here $U_-$ spans $E_u$ and $U_+$ spans $E_s$, so we have 
\[ p_-^{-1}(0) = E_0^* \oplus E_u^* = \Gamma_+, \quad p_+^{-1}(0) = E_0^* \oplus E_s^* = \Gamma_-.\]
Since $p_-\in C^{2-}$, the Hamiltonian vector field $H_{p_-}$ has $C^{1-}$ coefficients, which would a priori not be sufficient to define its flow as it is not Lipschitz. However, since this is a Hamiltonian vector field, the flow equation in local coordinates reads 
\begin{equation}\label{hamiltonfloweq} 
\dot{x}(t)=U_-(x(t)), \quad \dot{\xi}(t)=-\pl_{x}U_-(x(t)).\xi(t).
\end{equation}
Since $U_-\in C^{2-}$, we see that the first equation with initial condition $x(0)=x_0$ has a unique solution given by the (horocycle) flow $\phi_t(x_0)=e^{tU_-}(x_0)$ of $U_-$, 
and $(t,x)\mapsto \phi_t(x)$ is $C^{2-}$ but it is $C^2$ in the $t$-variable. Since $x(t)$ is well-defined (in a unique way), the second 
equation for $\xi(t)$ with $\xi(0)=\xi_0$ can now obviously be solved for each $x$ since it is linear, the solution is unique, and it is given by $\xi(t)=(d\phi_t(x)^{-1})^T\xi$.
We will thus define $e^{tH_{p_-}}$ to be the symplectic lift of $\phi_t$
\[ e^{tH_{p_-}}(x,\xi):= (\phi_t(x),(d\phi_t(x)^{-1})^T\xi)\]
which is the unique solution of \eqref{hamiltonfloweq} with initial condition $(x,\xi)\in T^*\mc{M}$.

Let $\delta > 0$ be small and define
\begin{equation}\label{defofU_delta}
\mathcal{U}_\delta := \{(x,\xi)\in T^*\mc{M} : |p_-(x,\xi)| < \delta,\, |p_+(x,\xi)| < \delta,\, |p(x,\xi)+1| < \delta\}.
\end{equation}

We now state a technical propagation Lemma,  which is comparable to \cite[Lemma 2.7]{Dy}.

\begin{figure}
             \centering
\begin{tikzpicture}[scale = 0.8, everynode/.style={scale=0.5}]
\tikzset{cross/.style={cross out, draw=black, minimum size=2*(#1-\pgflinewidth), inner sep=0pt, outer sep=0pt},
cross/.default={1pt}}


      	\draw[thick, ->] (-4.5, 0) -- (4.5, 0) node[right] {\small $p_+$};
		\draw[thick, ->] (0, -4.5) -- (0, 4.5) node[above] {\small $p_-$};

		\draw (-3,-3) rectangle (3, 3);
		\fill[gray, nearly transparent] (-3, -3) rectangle (3, 3);
		
		\draw (-1.5,-1.5) rectangle (1.5, 1.5);
		\fill[pattern=north east lines, pattern color=blue] (-1.5, -1.5) rectangle (1.5, 1.5); 
		
		\draw (-3,-3) rectangle (3, -1.5);
		\fill[pattern=north west lines, pattern color=red] (-3, -3) rectangle (3, -1.5);
		
		\draw (-3,3) rectangle (3, 1.5);
		\fill[pattern=north west lines, pattern color=red] (-3, 3) rectangle (3, 1.5);
		
		\draw (-2.25, -2.25) rectangle (2.25, 2.25);
		\fill[pattern=horizontal lines, pattern color=yellow] (-2.25, -2.25) rectangle (2.25, 2.25);
		
		\draw[thick, color=black!50!green] (-3.25,-3.25) rectangle (3.25, 3.25);
		
		\draw (1.5, -1.5) node[above left] {\small $\mathcal{U}_{\delta}$};
		\draw (2.25, -2.25) node[above left] {\small $\mathcal{U}_{\frac{3\delta}{2}}$};		
		\draw (3, -3) node[above left] {\small $\mathcal{U}_{2\delta}$};
		
		\fill (1.5, 0) node[above left] {\tiny $\delta$} circle (1pt);
		\fill (2.25, 0) node[above left] {\tiny $\frac{3\delta}{2}$}  circle (1pt);
		\fill (3, 0) node[above left] {\tiny $2\delta$}  circle (1pt);

\end{tikzpicture}
             \caption{\small The wavefront set setup in Lemma \ref{lemm:averageestimate}. In shaded blue we have $A = 1$, in red is $B = 1$, in yellow is $B_2 = 1$, in gray we have $B_1 = 1$, while the dark green rectangle bounds the set where $B_2' = 1$, all taken modulo $\mathcal{O}(h^\infty)$.}
             \label{fig:wfsetup}
\end{figure}
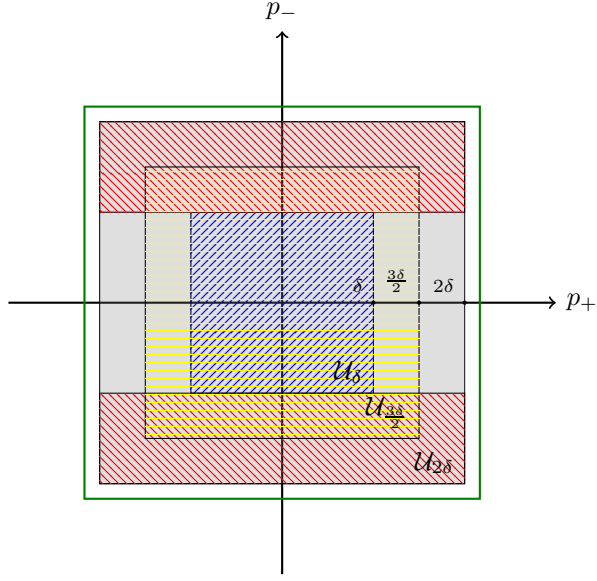

\begin{lemm}\label{lemm:averageestimate}
	Let $k \geq 1$. Fix $\delta > 0$ small and assume $A, B, B_1 \in \Psi_h^{\comp}(\mc{M})$ satisfy (see Figure \ref{fig:wfsetup})
	\begin{itemize}
		\item[1.] ${\rm WF}_h(A) \subset \mathcal{U}_{3\delta/2}$ and $A = 1 + \mc{O}(h^\infty)$ on $\overline{\mathcal{U}_{\delta}}$.
		\item[2.] ${\rm WF}_h(B) \subset \mathcal{U}_{3\delta} \cap \{|p_-| > \delta/2\}$ and $B = 1 + \mc{O}(h^\infty)$ on $\overline{\mathcal{U}_{2\delta}} \cap \{|p_-| \geq \delta\}$.
		\item[3.] ${\rm WF}_h(B_1) \subset \mathcal{U}_{3\delta}$ and $B_1 = 1 + \mc{O}(h^\infty)$ on $\overline{\mathcal{U}_{2\delta}}$.
	\end{itemize}
Let $W \in \Psi_{h, \rho, k}^{0}(\mc{M})$ with real principal symbol $\sigma(W)$ so that $\sigma(W)\to \sigma(W)_0\in C^0(T^*\mc{M})$ uniformly as $h\to 0$, and assume that for $T\gg 1$ and some $c>0$
	\begin{equation}\label{eq:averagecondition}
		\frac{1}{T} \int_0^T  (\sigma(W)_0 - V - \gamma) \circ {e^{-tH_p}} dt \geq c    \quad \mathrm{on}\,\,\,\, \mathcal{U}_\delta \cap \Gamma_+.
	\end{equation}
	Then for all $u_h \in \mc{D}'(\M)$ and all $N_1 \gg 1$, there is a $C > 0$ such that
	\begin{equation}\label{eq:ineq0}
		\lVert{Au}\rVert_{L^2} \leq Ch^{-1} \lVert{B_1(P_h(\la_h) - ihW)u}\rVert_{L^2} + C\lVert{Bu}\rVert_{L^2} + Ch^{N_1} \lVert{u}\rVert_{H_h^{-N_1}}.
	\end{equation}
\end{lemm}
\begin{proof} 
We split the proof in two steps.

\emph{Step 1.}
	Take $B_2 \in \Psi_h^{\comp}(\M)$ such that ${\rm WF}_h(B_2) \subset \mathcal{U}_{2\delta}$ and $B_2 = 1 + \mc{O}(h^\infty)$ on $\mathcal{U}_{3\delta/2}$. We claim that, for every fixed $\varepsilon_0 > 0$, there is $C>0$ such that for  all $u\in L^2(\M)$
	\begin{equation}\label{eq:claim1}
		\lVert{Au}\rVert_{L^2} \leq Ch^{-1} \lVert{B_2(P_h(\la_h) - ihW )u}\rVert_{L^2} + C\lVert{Bu}\rVert_{L^2} +\varepsilon_0 \lVert{u}\rVert_{L^2}.
	\end{equation}
	We prove this by contradiction: assume there is a sequence $u_{h} \in L^2(\M)$ with $h \to 0$ and $\|u_h\|_{L^2} = 1$, such that
	\begin{equation}\label{eq:contradictionassumption1}
		\|B_2(P_h(\la_h)- ihW )u_{h}\|_{L^2} = o(h), \quad \|Bu_{h}\|_{L^2} = o(1), \quad \|Au_{h}\|_{L^2} \geq \varepsilon_0.
	\end{equation}
	By \cite[Theorem E. 42]{DyZw} there is a semiclassical (Radon) measure $\nu$ associated to $u_h$. Now use propagation of singularities estimate Proposition \ref{prop:calcprops} (item 8): for every $(x, \xi) \in \mathcal{U}_{3\delta/2} \setminus \Gamma_+$, there is a $t_0 \geq 0$ with $e^{-t_0H_p}(x, \xi) \in \{\sigma(B) = 1\}$ and $e^{-tH_p}(x, \xi) \in \mathcal{U}_{3\delta/2}$ for $t \in [0, t_0]$. Thus for all $Q \in \Psi^{\rm comp}_h(\M)$ with ${\rm WF}_h(Q)\subset \mathcal{U}_{3\delta/2}\setminus \Gamma_+$,
	\[\|Qu\|_{L^2} \leq C \|Bu\|_{L^2} + Ch^{-1} \|B_2(P_h(\la_h) - ihW)u\|_{L^2} + \mc{O}(h^\infty).\]
By \eqref{eq:contradictionassumption1} we obtain $\|Qu\|_{L^2} = o(1)$ and so $\nu = 0$ on $\mc{U}_{3\delta/2}\setminus \Gamma_+$.

Next, \eqref{eq:contradictionassumption1} implies $(P_h(\la_h) - ihW)u = o(h)$ microlocally in $\mathcal{U}_{3\delta/2}$, so by Proposition \ref{prop:semiclassicalmeasurepropagation} we have for all $a \in C_0^\infty(\mathcal{U}_{3\delta/2})$
\[\int_{T^*\M} H_p a \, d\nu = 2\int_{T^*\M} (\sigma(W)_0-V-\gamma) a \, d\nu.\]
Set $f := 2 (\sigma(W)_0-V-\gamma)$. Equivalently, since for $t \geq 0$ we have $e^{-tH_p}: \mathcal{U}_{3\delta/2} \cap \Gamma_+ \to \mathcal{U}_{3\delta/2} \cap \Gamma_+$ and since $\nu$ is zero on $\mc{U}_{3\delta/2} \setminus \Gamma_+$, we may write an evolution equation for $\nu' = \nu|_{\mathcal{U}_{3\delta/2} \cap \Gamma_+}$,
\begin{equation}
	(e^{-tH_p})^* d\nu' = e^{\int_0^t f \circ e^{-rH_p} dr} d\nu', \quad t \geq 0.
\end{equation}
Applying this relation to $0 \leq a \in C_0^\infty(\mathcal{U}_{3\delta/2})$ and as $\nu$ is a Radon measure, we get for $t \gg 1$
\[\int_{\mathcal{U}_{3\delta/2} \cap \Gamma_+} a \circ e^{tH_p} d\nu = \int_{\mathcal{U}_{3\delta/2} \cap \Gamma_+} e^{\int_0^t f \circ e^{-rH_p} dr} a \, d\nu \geq e^{2c t} \int_{\mathcal{U}_{3\delta/2} \cap \Gamma_+} a \, d\nu,\] 
by the condition \eqref{eq:averagecondition}. The left hand side of this equation is bounded from above by $\|a\|_{L^\infty} \nu(\mathcal{U}_{3\delta/2})$, while the right hand side is growing exponentially fast. Thus, $\nu \equiv 0$ on $\mc{U}_{3\delta/2}$, contradicting the last point of \eqref{eq:contradictionassumption1} and proving the claim.

\emph{Step 2.} For every $B_2' \in \Psi_h^{\comp}(\M)$ with $B_2' = 1 + \mc{O}(h^\infty)$ microlocally on ${\rm WF}_h(B_2)$, we may apply inequality \eqref{eq:claim1} to $B_2'u$ to get\begin{equation}\label{eq:ineq1}
	\|Au\|_{L^2} \leq Ch^{-1} \|B_2(P_h(\la_h) - ihW)u\|_{L^2} + C \|Bu\|_{L^2} + \varepsilon_0 \|B_2'u\|_{L^2} + Ch^{N_1} \|u\|_{H_h^{-N_1}}.
\end{equation}
Choose $B_2'$ such that for every $(x, \xi) \in {\rm WF}_h(B_2')$, there is a $t_0 \geq 0$ with $e^{-t_0H_p} (x, \xi) \in \{\sigma(A) \neq 0\} \cup \{\sigma(B) \neq 0\}$, and for all $t \in [0, t_0]$ we have $e^{-tH_p}(x, \xi) \in \{\sigma(B_1) \neq 0\}$. Thus by the propagation of singularities estimate, Proposition \ref{prop:calcprops} (item 8), we obtain
\begin{equation}\label{eq:ineq2}
	\|B_2'u\|_{L^2} \leq Ch^{-1} \|B_1(P_h(\la_h) - ihW)u\|_{L^2} + C\|Bu\|_{L^2} + C \|Au\|_{L^2} + Ch^{N_1}\|u\|_{H_h^{-N_1}}.
\end{equation}
By using \eqref{eq:ineq2} to estimate the $\|B_2'u\|_{L^2}$ term in \eqref{eq:ineq1}, the elliptic estimate \cite[Theorem E.33]{DyZw} and the fact that ${\rm WF}_h(B_2) \subset {\rm WF}_h(B_1)$
\[\|Au\|_{L^2} \leq Ch^{-1} \|B_1(P_h(\la_h) - ihW) u\|_{L^2} + C\|Bu\|_{L^2} + \varepsilon_0 C \|Au\|_{L^2} + Ch^{N_1} \|u\|_{H_h^{-N_1}}.\]
Taking $\varepsilon_0$ small enough, we absorb the $\varepsilon_0 C \|Au\|_{L^2}$ term to the left hand side, thus completing the proof.
\end{proof}
We will now use this propagation estimate to deduce some regularity of the semiclassical measure $\mu$ in the direction of $H_{p_-}$. Recall the definition of $V_{\min}$, $V_{\max}$ in \eqref{eq:V_maxV_min}. 
\begin{prop}\label{regularitymu}
Assume $\gamma<\nu_{\min}-V_{\max}$ in \eqref{eq:A1}. If additionally $\beta\geq 2$ in \eqref{eq:A1}, for $\delta > 0$ small enough and any $a \in C_c^\infty(\mc{U}_{\delta})$, we have
 \begin{equation}\label{invariancep_-mu}
 \int_{\mc{U}_{\delta}} (H_{p_-}a - (-\divv U_-+2\alpha_V)a) d\mu=0,
 \end{equation}
where ${\rm div}$ denotes the divergence with respect to the contact measure $\alpha\wedge d\alpha$. Moreover, if \eqref{eq:A1} holds with $\|P_h(\lambda_h)u_h\|_{\mc{H}_h^{NG}} = \mc{O}(h^2)$ there is a $C>0$ such that
\begin{equation}\label{eq:weakhorocycleinvariance}
	\Big|\int_{\mc{U}_{\delta}} H_{p_-}(a)\, d\mu\Big| \leq C \|a\|_{L^\infty}.
\end{equation}
 \end{prop}
 \begin{proof}
Denoting $U_-^V := U_-+\alpha_V$, the commutation relation in Lemma \ref{[X,U_-^h]} reads:
\[ [h(X-V),h(U_-^V)^h]=-r_-^hh^2(U_-^V)^h +h^{2+\rho \beta_0}R_h, \quad R_h\in \Psi^1_{h,\rho}(\mc{M}),
\]
for some $\beta_0>0$.
Therefore, we have the relation
\begin{equation}\label{eq:commutationnew}
(P_h(\la_h) - ihr_-^h) h(U_-^V)^h = h(U_-^V)^h P_h(\la_h) - ih^{2+\beta_0\rho}R_h.
\end{equation} 
Next, for each $\eps>0$ so that $\nu_{\min}-V_{\max}-\gamma>3\eps$, there is $T_0>0$ such that for all $T>T_0$
\[\frac{1}{T}\int_0^T (r_- - V)\circ \varphi_{-t} 
\, dt > \nu_{\min}-V_{\max}-2\eps > \gamma+\eps.\] 
Thus, we can apply Lemma \ref{lemm:averageestimate} with the functions 
$h(U_-^V)^h u_h \in H_h^{-N - 1}(\M)$ and with $W:= r_-^h \in \Psi_{h, \rho, 2-\varepsilon}^0(\M)$, which gives
\begin{multline}\label{eq:horocyclicmicrolocalisation}
\|Ah(U_-^V)^h u_h\|_{L^2} \leq Ch^{-1} \|B_1(P_h(\la_h) - ihr_-^h)h(U_-^V)^h u_h\|_{L^2} + C \|Bh(U_-^V)^hu_h\|_{L^2}\\
 + \mc{O}(h^\infty) \|h(U_-^V)^h u_h\|_{H_h^{-N - 1}},
\end{multline}
where $A, B, B_1 \in \Psi_h^{\comp}(\M)$ satisfy the conditions of Lemma \ref{lemm:averageestimate}.
	 
We analyse \eqref{eq:horocyclicmicrolocalisation} by studying each term on the right hand side separately. Firstly, since ${\rm WF}_h(B)$ does not intersect  $\Gamma_+ \cap \{p = -1\}$, using the remark after Lemma \ref{lemm:support}, we have
\[\|Bh(U_-^V)^h u_h\|_{L^2} = o(h^{\beta-1}), \quad h \to 0.\]
Here we also used that $h(U_-^V)^h \in \Psi_{h, \rho, 0}^1(\M)$ is suitably bounded by Proposition \ref{prop:calcprops} (items 1 and 6). By the same boundedness properties 
\[\|h(U_-^V)^h u_h\|_{H_h^{-N - 1}} \leq C \|u_h\|_{H_h^{-N}} = \mc{O}(1),\]
which shows that the last term of \eqref{eq:horocyclicmicrolocalisation} equals $\mc{O}(h^\infty)$. Finally, for the first term it suffices to estimate, by \eqref{eq:commutationnew}
\begin{align*}
h^{-1}  \|B_1(h(U_-^V)^hP_h(\la_h) - ih^{2+ \rho\beta_0} R_h) u_h\|_{L^2} &\leq 
h^{-1} \|B_1h(U_-^V)^hP_h(\la_h) u_h\|_{L^2} + h^{1+\rho \beta_0}\|B_1 R_h u_h\|_{L^2}\\ 
&\leq h^{-1} \|P_h(\la_h) u_h\|_{H_h^{-N}}  + h^{1 + \rho \beta_0}\|u_h\|_{H_h^{-N}} + \mc{O}(h^\infty)\\
&= o(h^{\beta - 1}) + \mc{O}(h^{1 + \rho \beta_0}) + \mc{O}(h^\infty).
\end{align*}
We used that $\|R_h\|_{H^1_h\to L^2}=\mc{O}(1)$ by $\sigma(R_h)=\mc{O}(1)$ and Proposition \ref{prop:calcprops} (item 6). Therefore, by \eqref{eq:horocyclicmicrolocalisation}
\[\|Ah(U_-^V)^h u_h\|_{L^2} = o(h^{\beta - 1}) + \mc{O}(h^{1+ \rho \beta_0}),\]
and so $h(U_-^V)^h u_h = o(h)$ microlocally in $\mc{U}_\delta$ if  $\beta = 2$.
Now note that
\[h^{-1}\Im (-ih(U_-^V)^h) = \frac{-ih(U_-^V)^h - (-ih(U_-^V)^h)^*}{2i h} = \frac{1}{2}\divv U_-^h -\alpha_V^h= \mc{O}_{\Psi^0_{h,\rho,0}(\mc{M})}(1),\]	  
by Lemma \ref{l:approximation}. The main result \eqref{invariancep_-mu} then follows by Proposition \ref{prop:semiclassicalmeasurepropagation}, since $\divv U_-^h \to \divv U_-$, $\alpha_V^h\to \alpha_V$ and $H_{\sigma(-ih(U_-^V)^h)} a \to H_{p_-}a$ uniformly.
	  
Finally, if $\|P_h(\lambda_h)u_h\|_{\mc{H}_h^{NG}} = \mc{O}(h^2)$, a similar argument gives $h(U_-^V)^h u_h = \mc{O}(h)$ microlocally in $\mc{U}_\delta$. The final conclusion follows again by applying Proposition \ref{prop:semiclassicalmeasurepropagation}.
\end{proof}
We now show that Proposition \ref{regularitymu} implies some Lipschitz regularity of $\mu$ inside $\Gamma_+$. 

\begin{lemm}\label{lemm:lipschitz} 
Assume that $\gamma<\nu_{\min}-V_{\max}$ in \eqref{eq:A1}. If additionally $\beta\geq 2$ in \eqref{eq:A1}, then for $\delta > 0$ small enough, there is a $C > 0$ such that for every $\delta_0 > 0$ small enough
\begin{equation}\label{eq:lipschitz}
 \frac{\delta_0}{C}\leq \mu(\mathcal{U}_\delta \cap \{|p_+| < \delta_0\}) \leq C\delta_0.
\end{equation}
Moreover, if \eqref{eq:A1} is valid with $\|P_h(\lambda_h)u_h\|_{\mc{H}_h^{NG}} = \mc{O}(h^2)$, then the upper bound in \eqref{eq:lipschitz} holds.
\end{lemm}
\begin{proof}
First, we make the following observation using the contact structure $\alpha$:
we have 
\begin{equation}\label{OmegaU_+U_-}
H_{p_-}(p_+)(x,\xi) =\xi([U_-,U_+](x)) ,
\end{equation}
and $[U_-, U_+]$ is a $C^{1-}$ vector field that does not vanish such that $0 \neq \alpha([U_-, U_+]) = -d\alpha(U_-, U_+)$ (as $U_\pm\in \ker \alpha$), by the relation 
\begin{equation}\label{U_+U_-not=0}
0 \neq \alpha\wedge d\alpha(X, U_-, U_+) = 2d\alpha(U_-, U_+).
\end{equation}
This means that either $H_{p_-}(p_+)(x,\xi)>c_1$ or $-H_{p_-}(p_+)(x,\xi)> c_1$ for some $c_1>0$ on $\mc{U}_\delta$ if $\delta>0$ is small enough.
Let $\widetilde{a} \in C_c^\infty(\mathbb{R})$ with $\supp (\widetilde{a}) \subset (-2\delta_0, \delta_1)$ satisfying
\begin{equation}\label{eq:aproperties}
\|\widetilde{a}\|_\infty \leq 1, \quad \partial_s \widetilde{a} \geq -\frac{2}{\delta_1}, \quad \partial_s \widetilde{a} \geq \frac{1}{3\delta_0} \quad \mathrm{for} \quad |s| \leq \delta_0,
\end{equation}
for some small fixed $\delta_1 \in (0, \delta)$, independent of $\delta_0$ (here we take $\delta_0<\delta_1/2$). Let $\chi\in C_c^\infty(-\delta,\delta)$ be equal to $1$ in $(-\delta/2,\delta/2)$. Consider 
$a(x,\xi):=\til{a}(p_+(x,\xi))\chi(p_-(x,\xi))\chi(p(x,\xi)+1)$ which is a $C_c^{2-}(\mc{U}_\delta)$ function. 
If $H_{p_-}(p_+) >c_1$ on $\mc{U}_{\delta}$, using that $\supp(\mu)\subset (\Gamma_+\cap \{p=-1\})
\subset \{|p_-|\leq \delta/2, |p+1|\leq \delta/2\}$, we get
\[ \int_{\mc{U}_\delta} H_{p_-}(a)\, d\mu =\int_{\mc{U}_\delta} \partial_s\til{a}(p_+)H_{p_-}(p_+)\, d\mu\geq \frac{c_1}{3\delta_0}\mu(\{|p_+|\leq \delta_0\})-\frac{2c_1}{\delta_1}\mu(\{|p_+|\in (\delta_0,\delta_1)\}).\]
The result of Proposition \ref{regularitymu} (in the case $\|P_h(\lambda_h)u_h\|_{\mc{H}_h^{NG}} = \mc{O}(h^2)$ as well) also holds for $a\in C_c^{1}(\mc{U}_\delta)$ by using a density argument, so we can apply it to our function $a$ and, since $\mu(\{|p_+|\in (\delta_0,\delta_1)\})\leq \mu(\mc{U}_\delta)$, we get that there is a $C>0$ such that for all $\delta_0 > 0$ small
\[ \mu(\{|p_+|\leq \delta_0\})\leq C\delta_0.\]
In the case where $H_{p_-}(p_+)(x,\xi)<0$ we can do the same reasoning by bounding below the integral $-\int_{\mc{U}_\delta} H_{p_-}(a)\, d\mu$.

To obtain a lower bound for the measure of $\mc{U}_\delta\cap \{|p_+|\leq \delta_0\}$, we can proceed as follows. Without loss of generality, we can 
assume that $c_2>H_{p_-}(p_+)>c_1>0$ as above in $\mc{U}_\delta$ -- this precisely means that $p_+$ is increasing along $e^{tH_{p-}}$. Moreover, we compute $H_{p_-} p (x, \xi) = \xi([U_-, X](x)) = r_-(x) p_-(x, \xi)$, and thus both $p$ and $p_-$ are constant along $e^{tH_{p_-}}$ on $\Gamma_+$. 

We integrate \eqref{invariancep_-mu} to get, for each $a\in C_c^\infty(\mc{U}_{\delta/2})$ and $|t| \leq \frac{\delta}{2c_2}$ (so $\supp(a \circ e^{tH_{p_-}}) \subset \mc{U}_{\delta}$)
\begin{equation}\label{eq:horocyclicinvarianceflow}
	\int_{\mc{U}_\delta} a\circ e^{tH_{p_-}}\, d\mu=\int_{\mc{U}_\delta} e^{\int_0^t(-\divv(U_-)+2\alpha_V)\circ e^{-sH_{p_-}}ds} a \, d\mu.
\end{equation}
We thus get, using that $|-\divv(U_-)+2\alpha_V|$ is uniformly bounded, that there is $C>0$ independent of $\delta_0$ such that
\begin{equation}\label{uperbound1} 
\int_{0}^{\frac{\delta}{2c_2}}\int_{\mc{U}_\delta} a\circ e^{tH_{p_-}}\, d\mu \,dt\leq C\int_{\mc{U}_\delta}a \, d\mu.
\end{equation}
We choose $a={\til a}(p_+)\chi(p+1)\chi(p_-)$ with $\chi$ as above, $\til{a}\in C_c^\infty((-\delta_0,\delta_0); \rr^+)$ satisfying $\til{a}=1$ on $\{ |p_+|\leq \delta_0/2\}$. For each $z\in \mc{U}_{\delta/2}$, the map $t\mapsto p_+(e^{tH_{p_-}}(z))$ is a $C^1$-diffeomorphism for $|t| \leq \frac{\delta}{2c_2}$ since $H_{p_-}(p_+)\in [c_1,c_2]$. Thus we can perform the change of variable $q=p_+(e^{tH_{p_-}}(z))$
\[\begin{split} 
\int_{0}^{\frac{\delta}{2c_2}}\int_{\mc{U}_{\delta/2}} a(e^{tH_{p_-}}(z)) \, d\mu(z)dt =& 
 \int_{\mc{U}_{\delta/2}}\int_{p_+(z)}^{p_+\big(e^{\frac{\delta}{2c_2}H_{p_-}}(z)\big)}\frac{\til{a}(q)}{H_{p_-}(p_+)(e^{t(q)H_{p_-}}(z))}dq d\mu(z)\\
 \geq & \frac{1}{c_2}\int_{\mc{U}_{\delta/2}}\int_{p_+(z)}^{p_+\big(e^{\frac{\delta}{2c_2} H_{p_-}}(z)\big)}\til{a}(q)dq d\mu(z)\\
 \geq &  \frac{\delta_0}{c_2}\mu(\underbrace{\{z\in \mc{U}_{\delta/2} \, | \, (-\delta_0/2,\delta_0/2)\subset [p_+(z), p_+\big(e^{\frac{\delta}{2c_2}H_{p_-}}(z)\big)]\}}_{A:=}).
 \end{split} \]
Since  $H_{p_-}p_+ \in [c_1,c_2]$ on $\mc{U}_\delta$, the set $A$ contains $p_+^{-1}(\big[\frac{\delta_0}{2} - \frac{\delta c_1}{2c_2},-\frac{\delta_0}{2}\big])$.
Note first that by the Lipschitz bound and Lemma \ref{lemm:support}, for any $\delta > 0$ we have $\mu (\mc{U}_\delta \setminus E_0^*) > 0$. Then the horocyclic invariance \eqref{eq:horocyclicinvarianceflow} implies that for all $R_1, R_2 \in \mathbb{R}$ small enough, the strip $S_{R_1, R_2} = \{z \in T^* \M \,|\, R_1 \leq p_- \leq R_2\}$ satisfies $\mu(S_{R_1, R_2}) > 0$. 
Thus for $\delta_0$ small enough, $\mu(A)\geq C'$ for some constant $C'>0$ depending only on $c_1$, $c_2$ and $\delta$ (and not on $\delta_0$). We therefore obtain, by combining with \eqref{uperbound1}, 
that there is $C''>0$ such that for all $\delta_0>0$ small 
\[   \mu(\mc{U}_\delta\cap \{|p_+|\leq \delta_0\})\geq\int_{\mc{U}_\delta}a \, d\mu \geq C''\delta_0,\]
concluding the proof.
\end{proof}

\section{Proof of the main theorem}
 
We proceed to the proof of the main result. Note that Theorem \ref{th:intro} is an immediate corollary of Theorem \ref{th:intro2} for the case $V = 0$.

\begin{proof}[Proof of Theorem \ref{th:intro2}] 
We proceed by contradiction. Assume that the estimates in \eqref{eq:semiclassicalbounds} do not hold, i.e. there is a sequence $h_n\to 0$, $\la_n\in \cc$ and $u'_n,f_n\in \mc{H}_h^{NG}$ with norm $\|f_n\|_{\mc{H}_h^{NG}}=1$, such that $(-ih_nX + ih_nV - i\lambda_n)u_n'=f_n$ satisfies $\|u_n'\|_{\mc{H}_h^{NG}}>Ch_n^{-2}$ for some $C > 0$ (resp. $\|u_n'\|_{\mc{H}_h^{NG}}>nh_n^{-2}$) if $h_n^{-1} \Re \lambda_n \in \mc{S}_0(\eps)$ (resp. $h_n^{-1} \Re \lambda_n \in \mc{S}_1(\eps)$). Up to extracting a subsequence and rescaling $h_n$, we can assume that $\Im(\la_n) \to 1$ and $h_n^{-1}\Re(\lambda_n) \to -\gamma$ as $n \to \infty$ with $-\gamma  \in \mc{S}_0(\eps)$ (resp. $-\gamma \in \mc{S}_1(\eps)$). Define $u_n:= u_n'/\|u_n'\|_{\mc{H}_h^{NG}}\in \mc{H}_h^{NG}$ which satisfies 
\[ \|P_{h_n}(\la_n)u_n\|_{\mc{H}_h^{NG}}=\frac{\|f_n\|_{\mc{H}_h^{NG}}}{\|u_n'\|_{\mc{H}_h^{NG}}} = \begin{cases}
\mc{O}(h_n^{2}), \quad -\gamma \in \mc{S}_0(\varepsilon),\\
o(h_n^{2}), \quad \,\,-\gamma \in \mc{S}_1(\varepsilon).
\end{cases}\] 
This means that we are in the case of \eqref{eq:A1} with $\|P_h(\lambda_h) u_h\|_{\mc{H}_h^{NG}} = \mc{O}(h^2)$ and $N \geq \frac{1}{2}$, if $-\gamma \in \mc{S}_0(\eps)$, or with $\beta=2$ and $N \geq 1$, if $-\gamma \in \mc{S}_1(\eps)$.

By Lemma \ref{lemm:semiclassicalmeasureexistence} there is a semiclassical measure $\mu$ associated to $u_h$, which by Lemmas \ref{lemm:semiclassicalmeasureproperties} and \ref{lemm:support} has support in $\Gamma_+\cap \{p=-1\}$. By Lemma \ref{lemm:semiclassicalmeasureproperties}, we have for every $a \in C_c^\infty(\mc{U}_\delta)$ and $t \geq 0$ (see \eqref{defofU_delta} for the definition of $\mc{U}_\delta$)
\[\begin{gathered}
 \int_{\mc{U}_\delta} a\circ e^{tH_p}\, d\mu=\int_{\mc{U}_\delta} e^{-2\gamma t-2\int_0^t 
V\circ e^{-rH_p}dr} a \, d\mu.  \end{gathered}\]
Therefore for all $\eps' > 0$ and $t \geq T_{\varepsilon'}$ large enough, we have $e^{-tH_p} (\Gamma_+ \cap \mc{U}_\delta) \subset \Gamma_+ \cap \mc{U}_\delta$ and
\begin{equation}\label{eq:invofmubyflow}
e^{-2t(\gamma+V_{\max}+\eps')} \mu(\Gamma_+ \cap \mc{U}_\delta)\leq \mu(e^{-tH_p} (\Gamma_+ \cap \mc{U}_\delta)) \leq e^{-2t(\gamma+V_{\min}-\eps')} \mu(\Gamma_+ \cap \mc{U}_\delta).
\end{equation}
Moreover, we claim that for each $\varepsilon' > 0$ and $t \geq T_{\varepsilon'}$ large enough
 \begin{equation}\label{eq:exponentialiteration}
 	\{|p_+| \leq e^{-(\nu_{\max} + \varepsilon')t}\delta\} \cap \Gamma_+ \cap \mathcal{U}_\delta \subset e^{-tH_p}(\Gamma_+ \cap \mathcal{U}_\delta) \subset \{|p_+| \leq e^{-(\nu_{\min} - \varepsilon')t}\delta\} \cap \Gamma_+ \cap \mathcal{U}_\delta.
 \end{equation}
 To see this, first note that $H_p(p_+)(x, \xi) = \xi([X,U_+](x)) = r_+(x) p_+(x, \xi)$ by \eqref{[X+V,U_-]} and thus
 \[p_+(e^{-tH_p}(x, \xi)) = e^{-\int_0^t r_+(\varphi_{-r}(x)) dr} p_+(x, \xi), \quad t \in \mathbb{R},\]
and we obtain, by \eqref{eq:muminmax2}, that for $t \geq T_{\varepsilon'}$ large enough and $(x, \xi) \in \mc{U}_\delta$
 \[ e^{-(\nu_{\max} + \varepsilon')t}|p_+|(x,\xi)\leq |p_+|(e^{-tH_p}(x, \xi)) \leq e^{-(\nu_{\min} - \varepsilon')t}\delta,\]
thus giving \eqref{eq:exponentialiteration}. If $-\gamma \in \mc{S}_1(\varepsilon)$, by \eqref{eq:exponentialiteration} and Lemma \ref{lemm:lipschitz} (note it is here that we use $-\gamma > -\nu_{\min} + V_{\max}$), we obtain that there is $C = C(\varepsilon') > 0$ such that for $\varepsilon' > 0$ and $t \geq T_{\varepsilon'}$
\begin{equation}\label{eq:sandwich}
	C^{-1}e^{-(\nu_{\max}+\varepsilon')t} \leq \mu(e^{-tH_p} (\Gamma_+ \cap \mc{U}_\delta)) \leq C e^{-(\nu_{\min} - \varepsilon')t}.
\end{equation}
Combining these inequalities with \eqref{eq:invofmubyflow}, we obtain for all $t \geq T_{\varepsilon'}$
\begin{equation}\label{eq:sandwich2}
	C^{-1}e^{(2\gamma+2V_{\min}-\nu_{\max}-3\varepsilon')t}\leq \mu(\Gamma_+ \cap \mc{U}_\delta) \leq C e^{(2\gamma+2V_{\max}-\nu_{\min} + 3\eps')t}.
\end{equation}
Since $-2\gamma < -\nu_{\max} + 2V_{\min} - 2\varepsilon$ we get a contradiction by choosing $\varepsilon'$ small and letting $t\to \infty$.

Next, if $-\gamma \in \mc{S}_0(\eps)$, by \eqref{eq:exponentialiteration} and Lemma \ref{lemm:lipschitz} we similarly get the upper bound of \eqref{eq:sandwich}, which combined with \eqref{eq:invofmubyflow} yields the upper bound in \eqref{eq:sandwich2}. Since $-2\gamma > -\nu_{\min} + 2V_{\max} + 2\eps$, we can choose $\varepsilon'$ small enough and by letting $t\to \infty$ this would force to have $\mu(\mc{U}_\delta) = 0$, contradicting Lemma \ref{lemm:support}. 

Finally, the classical estimates \eqref{eq:classicalbounds} follow by introducing a semiclassical parameter $h := |\Im(s)|^{-1}$ and applying the semiclassical estimates \eqref{eq:semiclassicalbounds}, as well as \eqref{eq:isotropic-anisotropic}.
\end{proof}

\appendix
\section{An exotic symbol class}\label{app:exoticcalc}

Let $\mc{M}$ be a closed $n$-manifold equipped with a Riemannian metric $g$, let $k\geq 0$ and $0 <\rho<1$. We use the usual notation $\cjg \xi\cjd =(1+|\xi|^2)^{1/2}$, $h>0$ will be a small semiclassical parameter, we let $\bbar{T}^*\mc{M}$ be the fiber radial compactification of $T^*\mc{M}$ as defined in \cite[Section E.1.3]{DyZw} and $H_p$ will denote the Hamiltonian vector field of $p\in C^\infty(T^*\mc{M})$. Given an operator $P$, we write $\Im P = \frac{P - P^*}{2i}$ for the imaginary part of $P$ and $\Re P = \frac{P + P^*}{2}$ for the real part; then $P = \Re P + i \Im P$. Recall that for $x\in \rr$ we write $x_+ = \max(x, 0)$.

Most of the results we gather in this appendix are simple extensions of classical results in semiclassical analysis that can be found in the books \cite{Zw} and \cite[Appendix E]{DyZw}. We shall only point out the main differences with our setting.

For each $m\in \rr$, we define the exotic pseudo-differential calculus $\Psi_{h,\rho, k}^m(\mc{M})$ by saying that $A\in \Psi_{h,\rho, k}^m(\mc{M})$ if its Schwartz kernel $K_A$ is in $\mc{O}_{C^N(\mc{M}\times \mc{M})}(h^N)$ for all $N>0$ outside a neighborhood of the diagonal, and near the diagonal can be written in local coordinates as 
\[ K_A(x,y)=(2\pi h)^{-n}\int_{\rr^n} e^{\frac{i}{h}(x-y)\xi}a(x,\xi)d\xi,\]
where the local symbols are in the class $S^{m}_{h,\rho,k}(\rr^{2n})$ defined by the property: 
$a\in S^{m}_{h,\rho,k}(\rr^{2n})$ if $a\in C^\infty(\rr^{2n})$ is an $h$-dependent function and satisfies in local coordinates (for some $C_{\alpha,\beta}$ uniform in $h$)
\begin{equation}\label{eq:exoticsymbolclass'}
	|\pl_x^{\alpha}\pl_\xi^{\beta}a(x,\xi)|\leq 
C_{\alpha,\beta}\cjg \xi\cjd^{m-|\beta|}h^{-\rho(|\alpha|-k)_+}.
\end{equation}
We notice that it is important here, for the calculus, that the loss of $h^{-\rho}$ happens only in the $x$-derivatives and not in the $\xi$ derivatives. First, observe the basic properties for $0\leq k'\leq  k$
\begin{equation}\label{basicprop}
\begin{gathered} 
a\in S^{m}_{h,\rho,k}(\rr^{2n}),\,  b\in S^{m'}_{h,\rho,k'}(\rr^{2n}) \Longrightarrow ab\in S^{m+m'}_{h,\rho,k'}(\rr^{2n}),\\
a\in S^{m}_{h,\rho,k}(\rr^{2n}) \Longrightarrow \pl^\alpha_xa\in h^{-\rho (|\alpha|-k)_+}S^{m}_{h,\rho,(k-|\alpha|)_+}(\rr^{2n}) , \,\, \pl^\alpha_\xi a\in S^{m-|\alpha|}_{h,\rho,k}(\rr^{2n}),\\
\forall j\in [0,k], \quad h^{j\rho}S^{m}_{h,\rho,k-j}(\rr^{2n})\subset S^{m}_{h,\rho,k}(\rr^{2n}).
\end{gathered}
\end{equation}
We define $S^m_{h,\rho,k}(T^*\mc{M})$ to be $C^\infty(T^*\mc{M})$ functions that, using local coordinates on $\mc{M}$, are in $S^m_{h,\rho,k}(\rr^{2n})$. First, one directly sees from the formula of symbols under a change of coordinates  \cite[Theorem 9.9]{Zw} that the symbol 
in local coordinates being in $S^{m}_{h,\rho,k}(\rr^{2n})$ is invariant by change of coordinates, and moreover there is a principal symbol map 
\[ \sigma: \Psi_{h,\rho, k}^m(\mc{M})\to S^{m}_{h,\rho,k}(T^*\mc{M})/hS^{m-1}_{h,\rho,k}(T^*\mc{M}).\]
Using local charts and a partition of unity, we fix a semi-classical quantization 
${\rm Op}_h:S^{m}_{h,\rho,k}(T^*\mc{M})\to \Psi_{h,\rho,k}^m(\mc{M})$, which satisfies 
\[ \sigma ({\rm Op}_h(a))=a \, \, {\rm mod}\, \, hS^{m-1}_{h,\rho,k}(T^*\mc{M}).\]

We first check that symbols in this class are closed under composition. Recall from \cite[Theorem 4.14]{Zw} that if $A\in \Psi^{m}_{h,\rho,k}(\rr^n)$ and $B\in \Psi^{m'}_{h,\rho,k}(\rr^n)$ have full symbol $a,b$ then $AB$ has full symbol (as an oscillatory integral)
\begin{equation}\label{asharpb}
a \# b(x,\xi) =(2\pi h)^{-n}\int_{\rr^{2n}} e^{-\frac{i}{h}(x'.\xi')}a(x,\xi+\xi')b(x+x',\xi)dx'd\xi'\end{equation} 
which has expansion 
\begin{equation}\label{eq:symbolcomposition}
	a \# b =\sum_{|\alpha|\leq N} \frac{(-ih)^{|\alpha|}}{\alpha!} \partial_\xi^\alpha a \partial_x^\alpha b +\mc{O}_{S_{h,\rho,0}^{m+m'-N-1}(\rr^{2n})}(h^{(N+1)(1-\rho)}).
\end{equation}
Here  $|\partial_\xi^\alpha a \partial_x^\alpha b|\leq C_{\alpha}h^{-|\alpha|\rho}\cjg \xi\cjd^{m+m'-|\alpha|}$ so that higher order terms in the expansion are higher powers of $h$ and of $\cjg \xi\cjd^{-1}$.  
\begin{lemm}
	Let $a \in S^{m_1}_{h, \rho, k}(\rr^{2n})$ and $b \in  S^{m_2}_{h, \rho, k}(\rr^{2n})$. Then $a \# b \in S^{m_1 + m_2}_{h, \rho, k}(\rr^{2n})$. 
\end{lemm}
\begin{proof}
This follows from \eqref{eq:symbolcomposition} and \eqref{basicprop}.
\end{proof}

For $A\in \Psi_{h,\rho,k}^{m}(\mc{M})$, we say that $(x_0,\xi_0)\in \bbar{T}^*\mc{M}$ is not in ${\rm WF}_h(A)$  if there is a small neighborhood $U$ of $(x_0,\xi_0)$ in $\bbar{T}^*\mc{M}$ so that 
the full local symbol of $A$ restricted to $U$ is in $h^{N}S^{-N}_{h,\rho,0}(U)$ for all $N>0$. We also define the elliptic set ${\rm ell}_h(A)$ of $A\in \Psi_{h,\rho,k}^m(\mc{M})$ to be the set of points $(x_0,\xi_0)\in \bbar{T}^*\mc{M}$ so that for a neighborhood $U$ of $(x_0,\xi_0)$ there is $c_0>0$ so that $\cjg \xi\cjd^{-m}|\sigma(A)(x,\xi)|\geq c_0$, for $(x, \xi) \in U$. We finally list some properties of the $\Psi_{h, \rho, k}(\M)$ calculus.

\begin{prop}\label{prop:calcprops}
The following properties hold:
\begin{itemize} \itemsep2pt

	\item[1.] Let $A \in \Psi_{h, \rho, k}^m(\mc{M})$. If $B, B' \in \Psi_h^{\comp}(\mc{M})$ with ${\rm WF}_h(B) \cap {\rm WF}_h(B') = \emptyset$, then $BAB' \in h^\infty \Psi_h^{\comp}(\mc{M})$.
	
	\item[2.] The principal symbol is well-defined as a map
	\[\sigma: \Psi^m_{h, \rho, k}(T^*\mc{M}) \to S^m_{h, \rho, k}(T^*\mc{M}) / 
	hS^{m-1}_{h, \rho, k}(T^*\mc{M})\]
	with kernel $h\Psi_{h,\rho,k}^{m-1}(\mc{M})$, 
	and it satisfies for any $A\in \Psi^m_{h, \rho, k}(\mc{M})$ and $B\in \Psi^{m'}_{h, \rho, k}(\mc{M})$
	\[\begin{gathered}
	\sigma(AB) = \sigma(A) \sigma(B) \,\,\,\mathrm{mod}\,\,\, h^{1-\rho}S^{m+m'-1}_{h,\rho,k}(T^*\mc{M}), \\
	\sigma(AB) = \sigma(A) \sigma(B) \,\,\,\mathrm{mod}\,\,\, hS^{m+m'-1}_{h,\rho,k-1}(T^*\mc{M}) \textrm{ if }k\geq 1.
	\end{gathered}\]
	
	\item[3.] If $A \in \Psi^{m_1}_{h, \rho, k}(\mc{M})$ and $B \in \Psi^{m_2}_{h, \rho, k}(\mc{M})$, then 
\[\begin{gathered}
\, [A, B]\in h^{1 - \rho} \Psi^{m_1 + m_2 - 1}_{h, \rho, k}(\mc{M}), \\
h^{-1}[A, B] \in \Psi_{h, \rho, k - 1}^{m_1 + m_2 - 1}(\mc{M}) \textrm{ if }k\geq 1  ,
\end{gathered}\]
and $\sigma(h^{-1}[A, B]) =  -i\{\sigma(A), \sigma(B)\}$.
\item[4.] If $P = -ihX$ for $X$ a vector field on $\M$, $\Theta \in \Psi^m_{h, \rho, k}(\M)$ and $H_p \sigma(\Theta) \in S^m_{h, \rho, k}(\M)$, where $p =\sigma(P)= \xi(X)$, then
$[P, \Theta] \in h\Psi^m_{h, \rho, k}(\M)$.
\item[5.]If $A \in \Psi_{h, \rho, k}^m(\mc{M})$, then $A^* \in \Psi_{h, \rho, k}^m(\mc{M})$ and \[\begin{gathered}
\sigma(A^*)=\bbar{\sigma(A)} \,\,\,\mathrm{mod}\,\,\, h^{1-\rho}S^{m+m'-1}_{h,\rho,k}(T^*\mc{M}),\\
\sigma(A^*)=\bbar{\sigma(A)} \,\,\,\mathrm{mod}\,\,\, hS^{m-1}_{h,\rho,k-1}(T^*\mc{M}) \textrm{ if }k\geq 1.\end{gathered}\]	
\item[6.]  Each $A \in \Psi^0_{h, \rho, 0}(\mc{M})$ is bounded $L^2 \to L^2$ and for each $\varepsilon > 0$
\[\|A\|_{L^2 \to L^2} \leq (1 + \varepsilon) \sup_{h, x, \xi} |\sigma_h(A)(x, \xi)| + \mc{O}_{\varepsilon}(h^{\infty}).\]
Moreover, for any $A \in \Psi_{h, \rho, k}^m(\mc{M})$ and any $s \in \mathbb{R}$ we have
\[\|A\|_{H_h^s \to H_h^{s-m}} \leq (1 + \varepsilon) \sup_{h, x, \xi}|\langle{\xi}\rangle^{-m} \sigma(A)| + \mc{O}_{\eps}(h^\infty).\]
\item[7.] Let $P \in \Psi_{h, \rho, k}^p(\M)$, $A \in \Psi^m_{h, \rho, k}(\M)$ and $B_1 \in \Psi^l_{h, \rho, k}(\M)$. Assume ${\rm WF}_h(A) \subset \Ell_h(P) \cap \Ell_h(B_1)$. Then for all $s\in\rr$, $N > 0$, and $u$ with $B_1Pu \in H^{s-p-l}_h(\M)$
\[\|Au\|_{H^{s-m}_h} \leq C\|B_1Pu\|_{H^{s - p - l}_h} + \mc{O}(h^\infty) \|u\|_{H_h^{-N}}.\]
\item[8.] Assume $k \geq 1$ and let $P \in \Psi_{h, \rho, k}^1(\M)$ with $\Re P \in \Psi^1_h(\M)$ and $\Im P \in h\Psi^0_{h, \rho, k}(\M)$. Denote $p := \sigma(P)$ and assume that for each $(x,\xi)\in {\rm WF}_h(A)\subset \bbar{T}^*\mc{M}$, there is $T>0$ such that $e^{-TH_p}(x,\xi)\in {\rm ell}_h(B)$ and $e^{-tH_p}(x,\xi)\in {\rm ell}_h(B_1)$ for $t \in [0, T]$. Then for each $u\in L^2$ with $Pu \in L^2(\M)$, and every $N > 0$ there is a $C > 0$ such that
\[\|Au\|_{L^2} \leq C\|Bu\|_{L^2} + Ch^{-1} \|B_1 P u\|_{L^2} + Ch^N \|u\|_{H_h^{-N}}.\]
\end{itemize}
\end{prop}
\begin{proof} 1. This follows from the composition formula \eqref{eq:symbolcomposition}.\\
2. This was discussed above.\\
3. From the composition formula \eqref{eq:symbolcomposition}, locally we have
	\[ a\# b-b\# a \sim h(D_\xi a \partial_x b - D_\xi b \partial_x a) + \frac{1}{2}h^2(D_\xi^2 a \partial_x^2 b - D_\xi^2 b \partial_x^2 a) + \frac{1}{6}h^2(D_\xi^3 a \partial_x^3 b - D_\xi^3 b \partial_x^3 a) + \dotso\]
	where, after taking the expansion to a high enough order, the remainder is in $h^NS^{m-N}_{h,\rho,0}(\rr^{2n})$ for some large $N>0$.
	By \eqref{basicprop}, all these terms are in $h^{1-\rho}S^{m_1+m_2-1}_{h,\rho,k}(\rr^{2n})\cap hS^{m_1+m_2-1}_{h,\rho,k-1}(\rr^{2n})$ if $k\geq 1$, and the principal symbol of $h^{-1}[A,B]$ is  
	$-i(\pl_\xi a \partial_x b-\pl_x a \pl_\xi b)$.\\	
4. This follows from the composition formula \eqref{eq:symbolcomposition}, the fact that $\partial_\xi^{\alpha} \sigma(P) = 0$ for $|\alpha| \geq 2$ and item 3 above.\\
5.  This follows from the fact that, if $A$ has full local symbol $a$ in local coordinates, the full symbol $a^*$ of $A^*$ is 
	\[a^*(x, \xi) \sim \sum_\alpha \frac{h^{|\alpha|}}{\alpha!} \partial_\xi^\alpha D_x^\alpha \overline{a(x, \xi)}.\]	
6. The argument is standard (see \cite[Theorem 4.5]{GrSj}) and we just make a brief summary. Let $a := \sigma(A) \in S^0_{h, \rho, k}(\mc{M})$.  For $M=\|a\|_\infty$, we can construct $B_0 \in \Psi^0_{h, \rho, 0}(\mc{M})$ with principal symbol $b_0:=\sqrt{(1 + \varepsilon)^2M^2-|a|^2}\geq \eps M>0$, so that
	\[C:=(1 + \varepsilon)^2 M^2 - A^*A = B_0^*B_0 + h^{1-\rho}R_0,\]
	for some $R_0=R_0^* \in \Psi_{h, \rho, 0}^{-1}(\mc{M})$ by item 2 (and the fact that $a \in S^m_{h, \rho, 0}$ and $a > 0$ implies $\sqrt{a} \in S^{m/2}_{h, \rho, 0}$). Next we can choose $B_1=B_0 + \tfrac{h^{1-\rho}}{2}{\rm Op}_h(b_0^{-1})^*R_0$ so that $B_1^*B_1=C - h^{2(1-\rho)}R_1$ with $R_1\in \Psi^{-2}_{h, \rho, 0}(\mc{M})$.
	 We iterate this procedure to find $B_N \in \Psi_{h, \rho, 0}^{0}(\mc{M})$ such that 
	 $(1 + \varepsilon)^2 M^2 - A^*A=B_N^*B_N+h^{(N+1)(1-\rho)}R_N$ with $R_N\in\Psi_{h,\rho,0}^{-N-1}(\mc{M})$. By Schur's lemma we have $\|R_N\|_{L^2\to L^2}=\mc{O}(h^{(N + 1)(1-\rho) - n})$ and thus for any $u \in L^2$ 
\[
\|Au\|_{L^2}^2 =  (1 + \varepsilon)^2 M^2\|u\|_{L^2}^2 - \|B_Nu\|_{L^2}^2 + \mc{O}(h^{2(N + 1)(1-\rho)-n}) \|u\|^2_{L^2}
\]
which shows the desired estimate by choosing $N$ large. Similarly the case of arbitrary $k$ follows and the Sobolev bound is an easy consequence of this.\\
7. This follows from the parametrix construction in the elliptic set, the main thing to notice is that if $B \in \Psi_{h, \rho, k}^m(\mc{M})$, $A \in \Psi_{h, \rho, k}^{l}(\mc{M})$ and ${\rm WF}_h(A) \subset \Ell_h(B)$, then
		\[\frac{\sigma(A)}{\sigma(B)} \in S_{h, \rho, k}^{l-m}(\mc{M}).\]
8. We follow the proof of \cite[Theorem E.47]{DyZw} and divide the proof into steps.
	
	\emph{Step 0: an escape function.} Fix $\beta \geq 0$. There is a $g \in C^\infty(\overline{T}^*\mc{M})$ with $\supp g \subset \Ell_h(B_1)$, such that
	\[g \geq 0, \quad g > 0 \,\,\mathrm{on}\,\, {\rm WF}_h(A), \quad H_p g \leq - \beta g, \]
	where the last condition holds outside $\Ell_h(B)$.
	
	\emph{Step 1.} Note that $g \in S^0(T^*\M)$ and define 
		\[G:= \Op_h(\langle{\xi}\rangle^s g) \in \Psi_h^s(\M), \quad {\rm WF}_h(G) \subset \Ell_h(B_1).\]
		We can that $u \in C^\infty(\M)$. If we write $f = Pu$,
		\[\Im \langle{f, G^*Gu}\rangle = \underbrace{\Im \langle{(\Re P) u, G^*Gu}\rangle}_{\mathrm{term\,\, T_1}} + \underbrace{\Re \langle{(\Im P) u, G^*Gu}\rangle}_{\mathrm{term\,\, T_2}}.\]
We will bound the two terms on the right separately.

\emph{Step 2: term $T_1$.} Since ${\rm Re}(P)\in \Psi_h^1(\mc{M})$ is non exotic, this step is exactly the same as in the proof of \cite[Theorem E.47]{DyZw} and we get
\[T_1 \leq (C_1 - \beta)h \|Gu\|_{L^2}^2 + Ch\|Bu\|^2_{H_h^s} + Ch^2 \|B_1 u\|^2_{H^{s-\frac{1}{2}}_h} + \mc{O}(h^\infty)\|u\|_{H_h^{-N}}^2.\]

\emph{Step 3: term $T_2$.} 
Write				
\[T_2 = \Re \langle{(\Im P) u, G^*Gu}\rangle = \langle{(\Im P)Gu, Gu}\rangle + \Re \langle{[G, \Im P]u, Gu}\rangle.\]
We estimate the two terms on the right hand side separately. Firstly
\[|\langle{(\Im P)Gu, Gu}\rangle| = h |\langle{(h^{-1}\Im P)}Gu, Gu\rangle| \leq C_2 h \|Gu\|_{L^2}^2 + \mc{O}(h^\infty) \|u\|_{H_h^{-N}}^2,\]
where we used the boundedness property (item 6) for the exotic operator $h^{-1} \Im P$.
For the second term, we need to deal with a commutator. Note first that
\[\Re G^*[G, \Im P] = h \Re G^* [G, h^{-1} \Im P] \in h^2 \Psi^{2s-1}_{h, \rho, k-1}(\M),\]
by item 3 above. As ${\rm WF}_h(G^*[G, \Im P]) \subset \Ell_h(B_1)$, by the elliptic estimate
\begin{align*}
|\langle{\Re (G^*[G, \Im P]) u, u}\rangle| &=  h^2|\langle{h^{-2}\Re (G^*[G, \Im P]) u, Yu}\rangle| + \mc{O}(h^\infty)\|u\|_{H_h^{-N}}^2\\
&\leq Ch^2 \|B_1 u\|^2_{H_h^{s - \frac{1}{2}}} + \mc{O}(h^\infty) \|u\|_{H_h^{-N}}^2,
\end{align*}
where $Y \in \Psi_h^0(\M)$ is such that $Y = 1 + \mc{O}(h^\infty)$ microlocally on ${\rm WF}_h(\Re (G^*[G, \Im P]))$ and ${\rm WF}_h(Y) \subset \Ell_h(B_1)$. Here we used the boundedness in the exotic class item 6, item 1, the elliptic estimate in the exotic class (item 7).
Adding the two estimates we finally obtain 
\[ |\Re\langle{(\Im P)u, G^*Gu}\rangle| \leq C_2h \|Gu\|_{L^2}^2 + Ch^2\|B_1 u\|_{H_h^{s - \frac{1}{2}}}^2 + \mc{O}(h^\infty) \|u\|_{H_h^{-N}}^2.\]

\emph{Step 4.} Adding the estimates in Steps 2 and 3, we obtain
\[ \Im \langle{f, G^*Gu}\rangle \leq (C_1 + C_2 - \beta)h \|Gu\|_{L^2}^2 + Ch\|Bu\|_{H_h^s}^2 + Ch^2 \|B_1u\|_{H_h^{s - \frac{1}{2}}}^2 + \mc{O}(h^\infty) \|u\|^2_{H_h^{-N}}.\]
By ellipticity of $B$ on ${\rm WF}_h(G)$, there is $Q \in \Psi_h^{s}(\M)$ such that $G = QB_1 + R$, where $R \in h^\infty\Psi_h^{-\infty}(\M)$, thus 
\[|\langle{f, G^*Gu}\rangle| \leq |\langle{QB_1f, Gu}\rangle| + |\langle{Rf, Gu}\rangle| \leq C\|B_1f\|_{H_h^s} \|Gu\|_{L^2} + \mc{O}(h^\infty) \|u\|^2_{H_h^{-N}}.\]			
Now choose $\beta = C_1 + C_2 + 1$ to get
\[\|Gu\|_{L^2}^2 \leq C\|Bu\|^2_{H_h^s} + Ch^{-1} \|B_1f\|_{H_h^s} \|Gu\|_{L^2} + Ch\|B_1u\|^2_{H_h^{s-\frac{1}{2}}} + \mc{O}(h^\infty)\|u\|_{H_h^{-N}}^2.\]
We can absorb the $\|Gu\|_{L^2}$ term to the left hand side, at the cost of the additional term $Ch^{-2} \|B_1f\|_{H_h^s}^2$ on the right hand side.
Next, use the condition ${\rm WF}_h(A) \subset \Ell_h(G)$ and elliptic estimates to derive
\begin{equation}\label{inductionstart}
\|Au\|_{H^s_h} \leq C \|Bu\|_{H_h^s} + Ch^{-1} \|B_1f\|_{H_h^s} + Ch^{\frac{1}{2}} \|B_1 u\|_{H_h^{s - \frac{1}{2}}} + \mc{O}(h^\infty) \|u\|_{H_h^{-N}}.
\end{equation}

\emph{Step 5.} Here, one can use the same induction procedure as in \cite[Proof of Th. E.47]{DyZw} to show that for each $\ell\in \nn$
\[\|Au\|_{H^s_h} \leq C \|Bu\|_{H_h^s} + Ch^{-1} \|B_1f\|_{H_h^s} + Ch^{\frac{\ell}{2}} \|B_1 u\|_{H_h^{s - \frac{\ell}{2}}} + \mc{O}(h^\infty) \|u\|_{H_h^{-N}},\]
where the first step of the induction is exactly where we arrived in \eqref{inductionstart}.
\end{proof}

Next we discuss semiclassical defect measures in the setting of the exotic calculus.

\begin{prop}\label{prop:semiclassicalmeasureexistence}
Assume that $u_h\in L^2(\M)$ is a family satisfying $\|u_h\|_{L^2} = \mc{O}(1)$. Then there exists a Radon measure $\mu$, called semiclassical measure, and a sequence $h_j \to 0$, such that for any $A \in \Psi^{\comp}_{h, \rho, 0}(\M)$ with $\lim_{h\to 0} \sigma(A)(h;x,\xi) = a_0(x,\xi)$ in $C_c^0(T^*\M)$, it holds that
\[\lim_{j\to \infty}\langle{Au_{h_j}, u_{h_j}}\rangle_{L^2} = \int_{T^*\M} a_0\, d\mu.\]
\end{prop}
\begin{proof} We follow the proof of \cite[Theorem E.42]{DyZw}. By Proposition \ref{prop:calcprops}, we have that $A-{\rm Op}_h(a_h)\in h\Psi^{\rm comp}_{h,\rho,0}(\mc{M})$ for some symbol $a_h\in S_{h,\rho,0}^{\rm comp}(T^*\mc{M})$ so that $a_{h}\to a_0$ in $C_c^0(T^*\M)$; it suffices to prove the claim for $A = \Op_{h_j}(a_{h_j})$. We write $I_h(a_h) := \langle{{\rm Op}(a_h) u_h, u_h}\rangle_{L^2}$
and claim 
\begin{equation}\label{eq:limsupbound}
\limsup_{h \to 0} |I_h(a_h)| \leq C \limsup_{h \to 0} \|a_h\|_\infty \leq C\sup_{h} \|a_h\|_\infty.
\end{equation}
Indeed, by Cauchy-Schwarz and Proposition \ref{prop:calcprops} (item 6), we have
\begin{equation}\label{eq:Ibound}
|I_h(a_h)| \leq C \|a_h\|_{\infty} + \mc{O}_{A}(h^{\infty}),
\end{equation}	
where $C = C(\sup_h \|u_h\|_{L^2}) > 0$. Take a countable, dense subset $(a_h^{\ell})_{\ell\in \nn} \subset S^{\comp}_{h, \rho, 0}(T^*\mc{M})$, where $S^{\comp}_{h, \rho, 0}(T^*\mc{M})$ is equipped with the inductive limit topology from the seminorms in \eqref{eq:exoticsymbolclass'}. By a diagonal argument and since by \eqref{eq:Ibound} $I_h(a_h^\ell)$ is bounded for all $\ell$, we may extract a sequence $h_j \to 0$ such that $I_{h_j}(a_{h_j}^\ell)$ converges for all $\ell$. For each $a_h \in S^{\comp}_{h, \rho, 0}(T^*\mc{M})$ and $\ell$, we get by \eqref{eq:Ibound}
	\[\limsup_{j, j' \to \infty} |I_{h_j}(a_{h_j}) - I_{h_{j'}}(a_{h_{j'}})| \leq \limsup_{j, j' \to \infty} |I_{h_j}(a_{h_j}^\ell) - I_{h_{j'}}(a_{h_{j'}}^\ell)| + C\sup_{h} \|a_h - a_{h}^\ell\|_{\infty}.\]
	Using the density of $a^\ell_h$, we obtain that $I_{h_j}(a_{h_j})$ is a Cauchy sequence and we may define for $a_h\in S^{\comp}_{h, \rho, 0}(T^*\mc{M})$
	\[I(a_h) := \lim_{j \to \infty} I_{h_j}(a_{h_j}).\]
By \eqref{eq:Ibound}, the map $I$ satisfies for each $a_h\in S^{\rm comp}_{h,\rho,0}(T^*\mc{M})$
\begin{equation}\label{eq:IboundC0}
|I(a_h)| \leq C \limsup_{j \to \infty} \|a_{h_j}\|_{\infty} \leq C \sup_{h \in (0, h_0')} \|a_h\|, 
\end{equation}
	for any $h_0' > 0$. 	In particular, $I$ extends to a continuous linear functional on $C_c^0(T^*\mc{M})$ ($h$-independent functions). Given $a_h \in S^{\rm comp}_{h,\rho,0}(T^*\mc{M})$ with $\lim_{h\to 0} a_{h}=a_0 \in C_c^\infty(T^*\M)$ in the $C_c^0(T^*\mc{M})$ topology, we get by \eqref{eq:IboundC0} 
	\[|I(a_0 - a_h)| \leq C\sup_{h \in (0, h_0')} \|a_0 - a_h\|_{\infty} \to 0 \textrm{ as }h_0'\to 0,\]
and thus $I(a_h)=I(a_0)$. By G\aa rding's inequality $I(a_0) \geq 0$ when $a_0 \geq 0$ and so by the Riesz-Markov representation theorem there is a Radon measure $\mu$ such that for each $a_0 \in C_c^\infty(T^*\mc{M})$
\[I(a_0) = \lim_{j\to \infty}\langle{\Op_{h_j}(a_0) u_{h_j}, u_{h_j}}\rangle_{L^2}= \int_{T^*\mc{M}} a_0\,  d\mu.\]
The main claim follows from this by using $I(a_h) = I(a_0)$ under the given assumptions.
\end{proof}

Now we prove a version of a propagation estimate for the semiclassical measure in the exotic calculus. This is a slight extension of \cite[Theorem E.44]{DyZw}. 
Note that if $\sigma(P)$ is real valued and $k\geq 1$, then $\Im P \in h\Psi_{h, \rho, k-1}^{m-1}(\mc{M})$ if $P \in \Psi_{h, \rho, k}^m(\mc{M})$ by Proposition \ref{prop:calcprops} (item 5).

\begin{prop}\label{prop:semiclassicalmeasurepropagation}
Assume $\|u_h\|_{L^2} = \mc{O}(1)$ and $u_h$ converges to a semiclassical  measure $\mu$. Let $P \in \Psi^m_{h, \rho, k}(\M)$ with $k\geq 1$, denote $p:= \sigma(P)$ and assume that $p$ is real-valued for all $h$, and define $b := \sigma(h^{-1}\Im P)$. Assume that for each $a \in C_c^\infty(T^*\M)$
\[(H_p a)_0 = \lim_{h \to 0} H_p a \quad\mathrm{ and }\quad  b_0 = \lim_{h \to 0} b \quad \mathrm{ exist\,\, in }\quad C^0_c(T^*\mc{M}).\] 
Then there is $C>0$ such that for all $a \in C_c^\infty(T^*\M)$ and  $Y \in \Psi_h^{\comp}(\M)$ with $Y = 1 + \mc{O}(h^\infty)$ microlocally on $\supp(a)$
\[\Big|\int_{T^*\M} ((H_{p}a)_0 + 2b_0a) d\mu\Big| \leq C \|a\|_\infty \limsup_{h \to 0} (h^{ - 1} \|YPu_h\|_{L^2} \|Yu_h\|_{L^2}).\]
\end{prop}
\begin{proof}
	Assume without loss of generality that $a$ is real valued. Let  $A \in \Psi_h^{\comp}(M)$ be such that $\sigma(A) = a$ and $A^* = A$. We compute
	\begin{align*}
		h^{-1} \Im \langle{Pu_h, Au_h}\rangle &= (2i)^{-1}h^{- 1}\langle{APu_h, u_h}\rangle - (2i)^{-1} h^{- 1}\langle{P^*Au_h, u_h}\rangle\\
		&=(2i)^{-1}  \langle{h^{ - 1}[A, P]u_h, u_h}\rangle + \langle{(h^{-1}\Im P) Au_h, u_h}\rangle.
	\end{align*}
Now by Proposition \ref{prop:calcprops} (item 3), we have $h^{-1}[A, P] \in \Psi^{\comp}_{h, \rho, k-1}(\M)$ with $-i\sigma(h^{- 1}[A, P]) = H_p a$ and by assumptions $\sigma((h^{-1}\Im P) A) = ba$.  Thus by Proposition \ref{prop:semiclassicalmeasureexistence} there is a semiclassical measure $\mu$ such that the right hand side converges to, after extracting a subsequence $h_j \to 0$
\[\int_{T^*\M} \Big(\frac{1}{2}(H_{p}a)_0 + b_0a\Big) d\mu.\]
Moreover, the left hand side equals (using Proposition \ref{prop:calcprops}, item 1)
\[(2ih)^{-1}(\langle{PYu_h, AYu_h}\rangle - \langle{AYu_h, PYu_h}\rangle) + \mc{O}(h^\infty),\]
from which we easily deduce the main estimate.
\end{proof}

\end{document}